\documentclass{amsart}
\usepackage[utf8]{inputenc}

\usepackage{amsmath,amsfonts, amssymb}
\usepackage{graphicx}
\usepackage{xcolor}
\usepackage{enumerate}
\usepackage{cite}
\usepackage{tabu}

\usepackage{stmaryrd} 

\usepackage{quiver}
\usepackage{tikz}

\usepackage{fullpage}
\usepackage[active]{srcltx}

\usepackage{changepage}

\usepackage{chngcntr}

\usepackage{multibib}

\usepackage{url}
\usepackage{hyperref}
\hypersetup{
	colorlinks,
	linkcolor={red!50!black},
	citecolor={blue!50!black},
	urlcolor={blue!80!black}
}  

\theoremstyle{plain}
\newtheorem{theorem}{Theorem}[section]
\newtheorem*{thm:chfunction}{Theorem~\ref{thm:chfunction}}
\newtheorem*{thm:chfunctionext}{Theorem~\ref{thm:chfunctionext}}
\newtheorem*{thm:bothbounds}{Theorem~\ref{thm:bothbounds}}
\newtheorem*{thm:bothboundslinks}{Theorem~\ref{thm:bothboundslinks}}

\newtheorem{cor}[theorem]{Corollary}
\newtheorem{conj}[theorem]{Conjecture}

\newtheorem{lemma}[theorem]{Lemma}

\newtheorem{question}[theorem]{Question}

\newtheorem*{cor:tnoffiberedlinks}{Corollary~\ref{cor:tnoffiberedlinks}}
\newtheorem*{conj:gmn=g}{Conjecture~\ref{conj:gmn=g}}

\theoremstyle{definition}

\newtheorem{remark}[theorem]{Remark}

\newtheorem{observation}[theorem]{Observation}

\newcommand{\hatF}{\ensuremath{{\widehat{F}}}}

\newcommand{\calA}{\ensuremath{{\mathcal A}}}

\newcommand{\calT}{\ensuremath{{\mathcal T}}}

\newcommand{\calD}{\ensuremath{{\mathcal D}}}

\newcommand{\calH}{\ensuremath{{\mathcal H}}}

\newcommand{\comment}[1]{}

\newcommand{\bdry}{\ensuremath{\partial}}

\newcommand{\nbhd}{\ensuremath{\mathcal{N}}}

\newcommand{\N}{\ensuremath{\mathbb{N}}}
\newcommand{\Q}{\ensuremath{\mathbb{Q}}}
\newcommand{\R}{\ensuremath{\mathbb{R}}}
\newcommand{\Z}{\ensuremath{\mathbb{Z}}}

\newcommand{\MN}{{\ensuremath{\rm{MN}}}}
\newcommand{\TN}{{\ensuremath{\rm{TN}}}}

\newcommand{\cut}{\ensuremath{\backslash}}

\newcommand{\dcsum}{\ensuremath{\mathbin{\rotatebox[origin=c]{-30}{$\asymp$}}}} 

\definecolor{amaranth}{rgb}{0.9, 0.17, 0.31} 
\definecolor{carrotorange}{rgb}{0.93, 0.57, 0.13} 
\definecolor{citrine}{rgb}{0.89, 0.82, 0.04} 
\definecolor{dartmouthgreen}{rgb}{0.05, 0.5, 0.06} 
\definecolor{ballblue}{rgb}{0.13, 0.67, 0.8} 
\definecolor{ceruleanblue}{rgb}{0.16, 0.32, 0.75} 
\definecolor{amethyst}{rgb}{0.6, 0.4, 0.8} 
\definecolor{amber}{rgb}{1.0, 0.75, 0.0} 
\definecolor{burlywood}{rgb}{0.87, 0.72, 0.53} 

\title{Morse-Novikov numbers, tunnel numbers, and handle numbers of sutured manifolds}
\author{Kenneth L. Baker}
\address{Department of Mathematics, University of Miami, Coral Gables, FL 33146, USA}
\email{k.baker@math.miami.edu}

\author{Fabiola Manjarrez-Guti\'errez}
\address{Instituto de Matem\'aticas, Universidad Nacional Aut\'onoma de Mexico, Cuernavaca, Mor., MEXICO}
\email{fabiola.manjarrez@im.unam.mx}


\begin{document}

\begin{abstract}

Developed from geometric arguments for bounding the Morse-Novikov number of a link in terms of its tunnel number, we obtain upper and lower bounds on the handle number of a Heegaard splitting of a sutured manifold $(M,\gamma)$ in terms of the handle number of its decompositions along a surface representing a given 2nd homology class. 
Fixing the sutured structure $(M,\gamma)$, this leads us to develop the handle number function $h \colon H_2(M,\bdry M;\R) \to \N$ which is bounded, constant on rays from the origin, and locally maximal.  Furthermore, for an integral class $\xi$, $h(\xi)=0$ if and only if the decomposition of $(M,\gamma)$ along some surface representing $\xi$ is a product manifold.
\end{abstract}

\maketitle

\tableofcontents

\section{Introduction}

Let $L$ be an oriented link in $S^3$ with Morse-Novikov number $\MN(L)$ and tunnel number $\TN(L)$.  
Letting $M= S^3 \cut \nbhd(L)$ be the link exterior, $2\TN(L)+2$  counts the minimum number of critical points of a Morse function on $M$ for which $\bdry M$ is the maximal level while
 $\MN(L)$ counts the minimum number of critical points among regular circle valued Morse functions on $M$ dual to the meridional class.
In \cite{P} A. Pajitnov proved that $\MN(L)\leq 2\TN(L)$ using  a perturbation argument to construct a circle-valued Morse map from a given real valued Morse function on the link exterior.
Discussions about a constructive geometric argument for this upper bound on $\MN(L)$ led to developing a lower bound and further generalizations for sutured manifolds.

\subsection{Bounds between linear and circular Heegaard splittings of sutured manifolds}
For our constructive geometric arguments, we shift from  Morse functions and circular Morse functions to handle decompositions, Heegaard splittings, and circular Heegaard splittings where we count handles instead of critical points.  

Given a connected sutured manifold $(M,\gamma)$ without toroidal sutures, a Heegaard surface is a connected decomposing surface $S$ for which $\bdry S$ is the union of the cores of the annular sutures $A(\gamma)$ so that $(M,\gamma)$ decomposed along $S$ yields a pair of connected compression bodies for which $S$ is the ``thick'' surface and any annular suture is a vertical boundary component. The surface $S$ defines a {\em linear} Heegaard splitting of $(M,\gamma)$.  

More generally, for a ``thin'' decomposing surface $F$ for $(M,\gamma)$ that meets every toroidal suture, let the ``thick'' surface $S$ be the union of a Heegaard surface for each component of the decomposition of $(M,\gamma)$ by $F$.  Then together $F$ and $S$ define a {\em circular} Heegaard splitting of $(M,\gamma)$.  Note that a linear Heegaard splitting is simply a circular Heegaard splitting with no thin surface, $F=\emptyset$.

Dropping $\gamma$ from the notation,  the {\em handle number} $h(M,F,S)$ of the splitting of $(M,\gamma)$ with the thin and thick surfaces $F$ and $S$ counts the minimum number of handles needed to create the compression bodies in the decomposition of $(M,\gamma)$ first along $F$ and then along $S$.   Then we define the handle number $h(M,\gamma,\xi)$ of a class $\xi \in H_2(M, \bdry M)$, to be the minimum of $h(M,F,S)$ among circular Heegaard splittings for which $[F]=\xi$.  Of course this is only defined for classes $\xi$ represented by decomposing surfaces that meet every toroidal suture, and we say such classes are {\em regular}.  This regularity is a kind of ``non-zero linking number''  condition.

When $(M,\gamma)$ has no toroidal sutures, every class is regular.  In particular $\xi=0$ is a regular class and the linear handle number, $h(M,\gamma, 0)$, is realized by a linear Heegaard splitting with $F = \emptyset$.

Moreover, as we require neither the thin surface $F$ nor the result of the decomposition of $(M,\gamma)$ along $F$ to be connected, our circular Heegaard splittings $(M,F,S)$ may actually be what are commonly called generalized Heegaard splittings or generalized circular Heegaard splittings according to whether or not the homology class of $F$ is trivial. For simplicity, we will just call them all Heegaard splittings, using the adjectives {\em linear}, {\em circular}, and {\em generalized} for emphasis where useful.  Indeed, note that a Heegaard splitting $(M,F,S)$ in which $[F]$ is a non-primitive homology class is necessarily a generalized circular Heegaard splitting.

Finally, we also define a sort of Euler characteristic complexity $\chi^{h}(M,\gamma,\xi)=\chi^{h}_{(M,\gamma)}(\xi)$ that measures the minimum ``size'' of thin surface needed in a Heegaard splitting that realizes the handle number of the class $\xi$.  
This complexity arises naturally when we bound the circular handle number from below by the linear handle number.

\begin{thm:bothbounds}
Suppose $(M, \gamma)$ is a sutured manifold  without toroidal sutures.
Then for any $\xi \in H_2(M,\bdry M)$ we have
\[ h(M,\gamma, \xi) \leq h(M,\gamma,0)-2d  \leq h(M,\gamma, \xi) + 2\chi^{h}_{(M,\gamma)}(\xi)\]
where $d=0,1,2$ depending on whether neither, just one, or both of $R_+(\gamma)$ and $R_-(\gamma)$ are empty.
\end{thm:bothbounds}

\begin{proof}
Theorem~\ref{thm:lineartocircular} develops the first inequality
while
Theorem~\ref{thm:generallowerbound} develops the second.
\end{proof}

As an application, let $(M,\gamma_+)$ be the sutured manifold where $M=S^3 \cut \nbhd(L)$ is the exterior of an oriented link $L$ in $S^3$ with $R_+(\gamma_+) = \bdry M$.  Let $\xi \in H_2(M,\bdry M)$ be the class of a Seifert surface for $L$.  Then $\MN(L) = h(M,\gamma_+, \xi)$ while $2\TN(L)+2 = h(M,\gamma_+,0)$.   Since $R_+(\gamma_+) = \bdry M \neq \emptyset$ while $R_-(\gamma_+) = \emptyset$, the bound 
\[h(M,\gamma, \xi) \leq h(M,\gamma,0)-2d \]
of Theorem~\ref{thm:bothbounds} reduces to Pajitnov's bound 
$\MN(L) \leq 2\TN(L)$, \cite{P}.
For the lower bound on $\MN(L)$ in this context, let us define the {\em Morse-Novikov genus} $g_{\MN}(L)$ to be the minimum genus among (possibly disconnected) Seifert surfaces $F$ for $L$ that are the thin surface of a Heegaard splitting realizing $\MN(L)=h(M,\gamma_+,\xi)$.
It turns out that in this situation $\chi^{h}(M,\gamma_+,\xi)$ is just $2g_{\MN}(L)$. Hence Theorem~\ref{thm:bothbounds} becomes
\begin{thm:bothboundslinks}
For an oriented link $L$ in $S^3$,   
\[ \MN(L) \leq 2\TN(L) \leq \MN(L) + 4g_{\MN}(L).\]
\end{thm:bothboundslinks}

\begin{cor:tnoffiberedlinks}
If $L$ is a fibered link, then $\TN(L) \leq 2g(L)$.
\end{cor:tnoffiberedlinks}

\begin{proof}
If $L$ is fibered, then $\MN(L) = 0$ and $g_\MN(L) = g(L)$ since there is a unique incompressible Seifert surface.
Then Theorem~\ref{thm:bothboundslinks} implies the result.
\end{proof}

One may expect that $g_{\MN}(L) = g(L)$, that the Morse-Novikov genus of an oriented link $L$ is realized by a minimal genus Seifert surface. However it is not so clear.  Nevertheless, by \cite{Baker-MN}, we do know the Morse-Novikov genus of a knot is realized by an incompressible Seifert surface.  That proof also extends to oriented links allowing for disconnected Seifert surfaces.
\begin{conj:gmn=g}[Cf. {\cite[Remark p440]{CTPforKnots}}]
For an oriented link $L$ in $S^3$, $g_{\MN}(L) = g(L)$.
\end{conj:gmn=g}

In Section~\ref{sec:handlecomplexities}, 
Question~\ref{ques:handlecharvsthurstonnorm} asks about similar relationships for more general sutured manifolds and homology classes.

\subsection{Handle indices}

Theorem~\ref{thm:bothbounds} has a somewhat cleaner statement using  {\em handle indices} rather than {\em handle numbers}.  Instead of adding up the handles of all dimensions, the handle index counts the number of $1$ and $2$ handles minus the number of $0$ and $3$ handles.  Effectively, the handle index $j(M,F,S)$ of a splitting is the handle number $h(M,F,S)$ minus twice the number of handlebodies in the decomposed manifold. It turns out that at most one $0$--handle and one $3$--handle is used in splittings realizing $h(M,\gamma, \xi)$, corresponding directly to whether $R_-(\gamma) = \emptyset$ or $R_+(\gamma)=\emptyset$; see Lemma~\ref{lem:compressionbodymeridianandverticaldecomp}.  Consequently $j(M,\gamma,\xi) = h(M,\gamma,\xi)$ for all $\xi\neq 0$ while $j(M,\gamma,0) = h(M,\gamma, 0)-2d$ for $d$ as in Theorem~\ref{thm:bothbounds}.  Hence Theorem~\ref{thm:bothbounds} may be rewritten as:
\begin{cor}\label{cor:handleindexbounds}
Suppose $(M,\gamma)$ is a sutured manifold without toroidal sutures.  Then for any $\xi \in H_2(M,\bdry M)$ we have
\[ j(M,\gamma, \xi) \leq j(M,\gamma,0) \leq j(M,\gamma,\xi) + 2\chi^{h}(M,\gamma,\xi) \]
\end{cor}
Written as such, this prompts one to examine the relationships among $h(M,\gamma, \xi)$ over all regular classes $\xi \in H_2(M,\bdry M)$.

\subsection{A function on homology}
The arguments that prove Theorem~\ref{thm:bothbounds} reveal interesting structure of $h(M,\gamma,\xi)$ as a function of the regular classes of $H_2(M,\bdry M)$.   Here we fix the sutured manifold $(M,\gamma)$ and write $h(\xi)$ for $h(M,\gamma,\xi)$, viewing it as a function on homology.

\begin{thm:chfunction}
For a sutured manifold $(M, \gamma)$,
the handle number function 
\[h \colon H_2(M,\bdry M) \to \N\]
has the following properties for any regular classes $\alpha, \beta \in H_2(M, \bdry M)$: 
\begin{enumerate}
    \item $h(\beta) \leq N$ for some constant $N$ (depending on $(M,\gamma)$),
    \item $h(k \beta) = h(\beta)$ for any integer 
    $k >0$, and
    \item $h(\alpha+m\beta) \leq h(\beta)$ for sufficiently large $m$.
\end{enumerate}
\end{thm:chfunction}

As the regular classes form a dense open set of $H_2(M,\bdry M)$ as shown in Lemma~\ref{lem:regclasses}, these properties enable a natural extension to a function on homology with real coefficients.

\begin{thm:chfunctionext}
The handle number extends to a function 
\[ h \colon H_2(M,\bdry M; \R) \to \N\]
that is
\begin{enumerate}
    \item bounded by constant,
    \item constant on rays from the origin,  and
    \item locally maximal.
\end{enumerate}
\end{thm:chfunctionext}

As discussed in Section~\ref{sec:fiberedfaces}, if $(M,\gamma)$ is a sutured manifold with $\bdry M$ toroidal and no annular sutures, then $h=0$ exactly on fibered faces of the Thurston norm.  More generally, for an integral class $\xi$, $h(\xi)=0$ if and only if the decomposition of $(M,\gamma)$ along some surface representing $\xi$ is a product manifold.

\subsection{Acknowledgements}

KLB was partially supported by a grant from the Simons Foundation (grant \#523883  to Kenneth L.\ Baker).

FM-G was partially supported by Grant UNAM-PAPIIT IN108320.

\section{Terminology}\label{sec:terminology}
Much of this can be found throughout the literature in various forms, though we take some liberty with resetting terminology so as to not be encumbered with modifiers.
We largely follow and broaden the perspective in \cite{Baker-MN} based on work of Goda \cite{Goda-circlevaluedmorsetheory}.

\subsection{Notations and conventions}
All manifolds are assumed to be compact and oriented. All submanifolds are assumed to  be properly embedded. If $Y$ is a $n$--manifold  and  $S\subset Y$ is a properly embedded submanifold we denote by $\nbhd(S)$ an open regular neighborhood around $S$, 
and by $\nbhd(\bdry S)$ the open regular neighborhood $\nbhd(S)\cap \bdry Y$ of $\bdry S$ in $\bdry Y$. 

If the codimension of $S$ in $Y$ is one, then  $\nbhd(S)$ is called a (bi)collar neighborhood of $S$ and is equivalent to $S\times (-\epsilon, \epsilon)$, for $\epsilon>0$. The product $S\times [-\epsilon, \epsilon]$, for $\epsilon>0$, is called a closed collar neighborhood of $S$, or simply a collar of $S$.

We use $Y\cut S$ to denote  $Y-\nbhd(S)$.

\subsection{Sutured manifolds and compression bodies}

Sutured manifolds were introduced by Gabai \cite{Gabai-FT3M-I} and provide a convenient framework for discussing compression bodies and various generalizations of Heegaard splittings, e.g. \cite{Goda-circlevaluedmorsetheory, Baker-MN}

A {\em sutured manifold} is a compact oriented $3$--manifold $M$ with a disjoint pair of subsurfaces $R_+$ and $R_-$ of $\bdry M$ such that
\begin{itemize}
    \item the orientation of $R_+$ is consistent with the boundary orientation of $\bdry M$ while the orientation of $R_-$ is reversed,
    \item $\bdry M \cut (R_+ \cup R_-)$ is a collection $\gamma$ of annuli $A(\gamma)$ and tori $T(\gamma)$, collectively called the {\em sutures}, and
    \item each annular suture joins a component of $\bdry R_+$ to a component of $\bdry R_-$.
\end{itemize}

Typically, sutured manifolds are written as $(M,\gamma)$ where $\gamma$ represents the sutured structure.  With this sutured structure understood, we also may use $\gamma$ to represent the core curves of $A(\gamma)$, oriented to be isotopic to $\bdry R_+$.  
At times it is convenient to use $\bdry_+ M$ and $\bdry_- M$ to refer to the surfaces $R_+$ and $R_-$.
(One may also care to use $\bdry_v M$, for {\em vertical} boundary), to refer to the annular sutures $A(\gamma)$.)
We may also use variations of $R_\pm(M)$ or $R_\pm(\gamma)$ for these surfaces $R_-$ and $R_+$.

Suppose $(S, \bdry S)$ and $(T,\bdry T)$ are two embedded oriented surfaces in general position in $(M, \bdry M)$. Define $T \dcsum S$ to be the {\em double curve sum of $T$ and $S$}, i.e., cut and paste along the curves of intersection to get an embedded oriented surface representing the cycle $T+S$. The surface $T \dcsum S$ is an embedded oriented surface coinciding with $T \cup S$ outside a regular neighborhood of the intersections $T \cap S$.  See \cite[Definition 1.1]{Scharlemann_SM}.

A {\em decomposing surface} $(S,\bdry S)\subset(M, \bdry M)$ is an oriented properly embedded surface which intersects each toroidal suture in coherently oriented parallel essential circles and each annular suture either in circles parallel to the core of the annulus or in essential arcs (not necessarily oriented coherently). There is a natural sutured manifold structure $(M', \gamma')$ on $M'= M\cut S$, where:
\begin{itemize}
    \item $\gamma'=(\gamma \cap M')\cup \nbhd(S_+\cap R_-)\cup \nbhd(S_- \cap R_+)$,
    \item $R_+'= ((R_+\cap M')\cup S_+)-int(A(\gamma'))$, and
    \item $R_-'= ((R_-\cap M')\cup S_-)-int(A(\gamma'))$.
\end{itemize}
Here $\nbhd(S) = S \times (-\epsilon, \epsilon)$ is a bicollar neighborhood of $S$ where 
$S=S\times \{0\}$ and we set $S_{\pm}= S\times \{\mp \epsilon \}$.  With this arrangement, the normal of the oriented surface  $S_+$ points out of $M\cut S$ and of $S_-$ points into $M'=M\cut S$.
We say $(M',\gamma')$ is obtained from a {\em sutured manifold decomposition} of $(M,\gamma)$ along $S$, and we denote this decomposition as $(M,\gamma) \overset{S}{\leadsto} (M',\gamma')$. See Figure~\ref{fig:SMdecomp}.

\begin{figure}
    \centering
    \includegraphics[height=2.5cm]{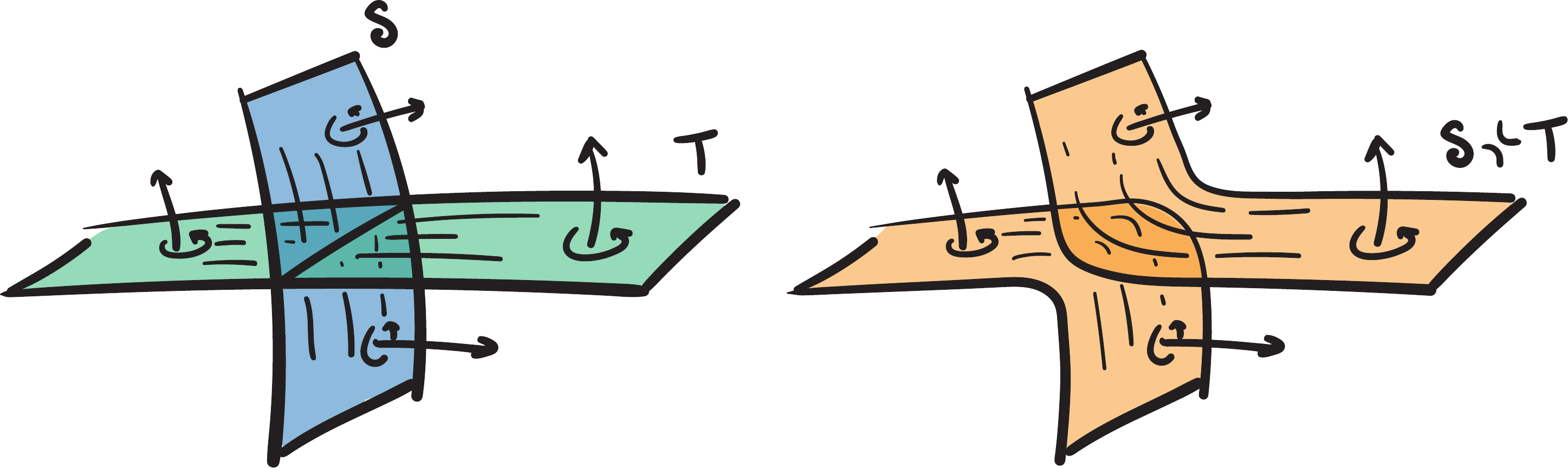}
    \caption{The double curve sum $S \dcsum T$ of the surfaces $S$ and $T$.} 
    \label{fig:doublecurvesum}
\end{figure}

\begin{figure}
    \centering
    \includegraphics[height=2.5cm]{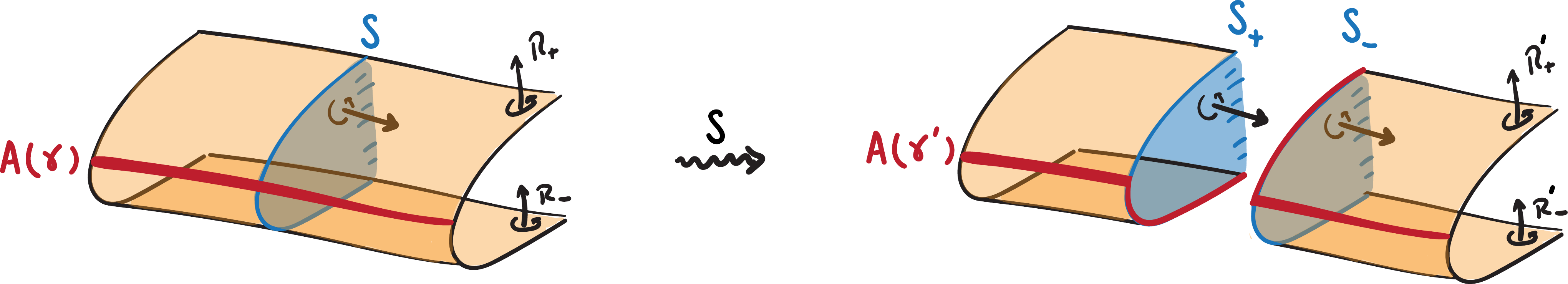}
    \caption{A sutured manifold decomposition along a surface $S$.}
    \label{fig:SMdecomp}
\end{figure}

\medskip
A {\em compression body} $C$ is a special kind of sutured manifold that may be formed from a collar of $R_-$ and a collection of $0$--handles by attaching $1$--handles.   The collar of $R_-$ is a product $R_- \times [-1,1]$ of a non-empty and possibly disconnected compact oriented surface $R$.  Then $1$--handles are attached to $R_- \times \{1\}$ and the boundaries of the $0$--handles.  
As a sutured manifold, the surface $\bdry_- C = R_-$ is identified with $R_- \times \{-1\}$, the annular sutures are $\bdry_v C = \bdry R_- \times [-1,1]$, and $\bdry_+ C = R_+$ is what remains of $\bdry C$.

A {\em meridional disk} of a compression body is a compressing disk for $R_+$.  In the above formulation of compression bodies, $R_-$ is necessarily incompressible in the compression body. For such a compression body $C$, we may say  that $\bdry_- C = R_-$ is the {\em thin boundary}, $\bdry_+ C = R_+$ is the {\em thick boundary}, and $\bdry_v 
C = A(\gamma)$ is the {\em vertical boundary}.

A {\em trivial} compression body is a compression body without any handles; hence it is a product.  A {\em product ball} is a trivial compression body homeomorphic to a $3$--ball when the sutured manifold structure is forgotten; it has the form $D^2 \times [-1,1]$.

A {\em handlebody} is a connected compression body in which $R_- = \emptyset$. The entire boundary is $R_+$, a closed oriented surface.  Hence, a handlebody may be viewed as a single $0$--handle with $g$ $1$--handles attached where $g$ is the genus of $R_+$.  We also say $g$ is the genus of the handlebody.  More generally one may regard a handlebody as a closed regular neighborhood of a connected finite graph in which the vertices give rise to the collection of $0$--handles and the edges determine the $1$--handles.  

We may ``dualize'' a compression body by swapping $R_+$ and $R_-$.
Then a {\em dualized compression body} may be regarded as a sutured manifold formed from a collar of $R_-$ by attaching $2$--handles and filling with $3$--handles a collection of resulting spherical boundary components that are disjoint from the vertical boundary.  In passing between a compression body and a dualized compression body, duality exchanges the $1$-- and $2$--handles, the $0$-- and $3$--handles, and the roles of the surfaces $R_+$ and $R_-$. So, dually, $R_+$ is the thin boundary, $R_-$ is the thick boundary, and a handlebody is a connected dualized compression body in which $R_+ = \emptyset$. See Figure~\ref{fig:compressioinbodydual}.

\begin{figure}
    \centering
    \includegraphics[width=.7\textwidth]{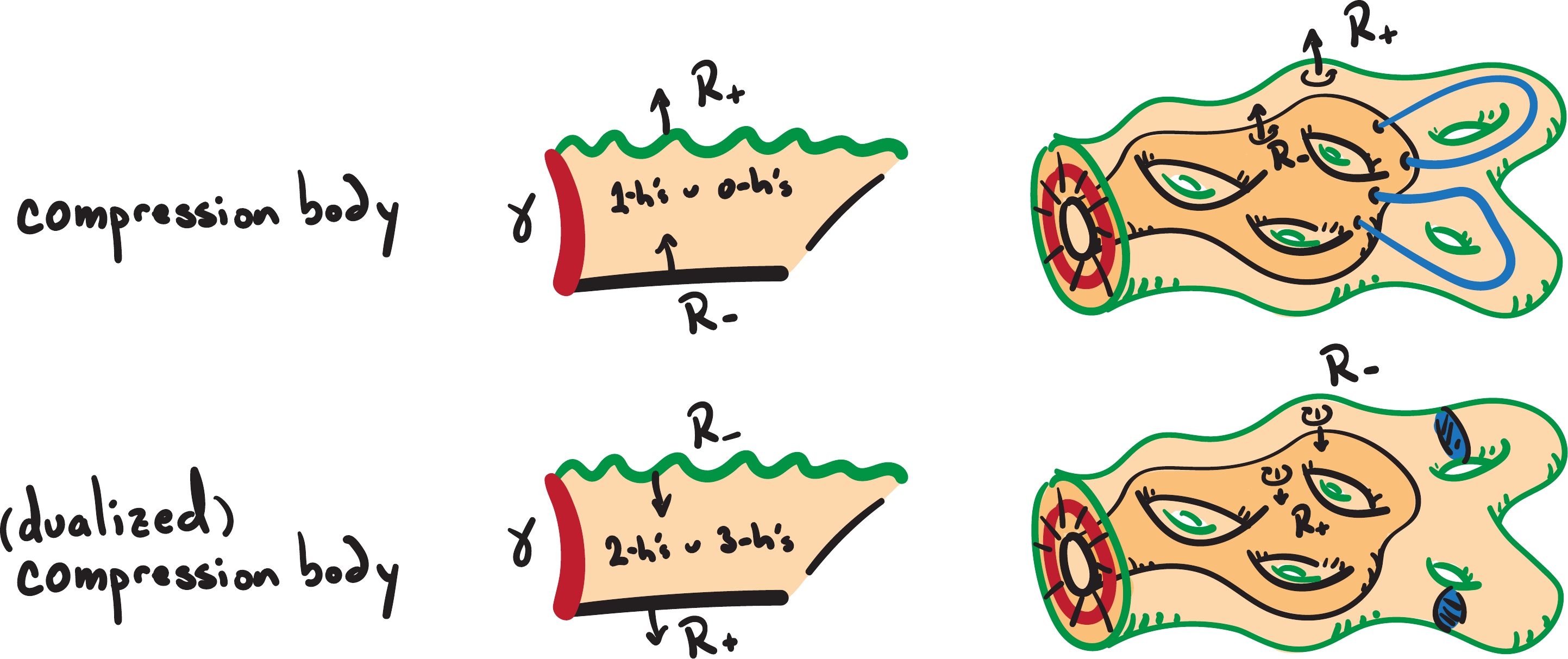}
    \caption{A compression body and a dualized compression body, each shown first with a schematic depiction and then with an example.  The examples also show the cores for a choice of $1$--handles and $2$--handles.}
    \label{fig:compressioinbodydual}
\end{figure}

\begin{remark}
Forgetting about the gradient flows implied by the indices of handles and the (normal) orientations of $R_+$ and $R_-$, a compression body $C$ can be viewed as a $3$--manifold  obtained by

\begin{itemize}
    \item (the ``Inside-Out'' construction) attaching $1$--handles to a collar of the surface $\bdry_- C$ and some $0$--handles so that no spherical boundary is left disjoint from $\bdry_v C$, or equivalently
    \item (the ``Outside-In'' construction)  attaching $2$--handles to a collar of the surface $\bdry_+ C$ and filling with $3$--handles a collection of spherical boundary components disjoint from $\bdry_v C$.
\end{itemize}

\end{remark}

\subsection{Handle numbers, Heegaard splittings }
\label{sec:handle}
The {\em handle number} $h(C)$ of a compression body $C$ is the minimum number of $0$-- and $1$--handles needed to form $C$.   The handle number of a handlebody $H$ of genus $g$ is $h(H)=g+1$. 
In the dualized setting,  the handle number is the minimum number of $2$-- and $3$--handles needed to form the compression body.

Let $(M,\gamma)$ be a sutured manifold without toroidal sutures.
A {\em linear Heegaard splitting} of $M$, denoted $(M,S)$ or $(A,B;S)$, is a decomposition of $M$ into a compression body $A$ and dual compression body $B$ along a surface $S = R_+(A) = R_-(B)$.  

Observe that $S$ is then a decomposing surface with $\bdry S$ contained in the annular sutures of $M$.
Then the sutured manifold decomposition $M' = M \cut S$ yields a disjoint union of a 
 compression body $A'$ and a 
dualized compression body $B'$ 
where $S_+ = R_+(A')$, $S_- = R_-(B')$, and $M = A' \cup \nbhd(S) \cup B'$.
In particular, $A= A' \cup S\times [-\epsilon, 0]$ and $B= B'\cup \times [0,\epsilon]$ so that $A\cong A'$ and $B \cong B'$ as sutured manifolds.
Note that $R_-(M) = R_-(A)$ and $R_+(M) = R_+(B)$.
See Figure~\ref{fig:HSsuturedman}.

The surface $S$ is called a {\em Heegaard surface} for $(M,\gamma)$. As $S$ is the thick boundary of both $A$ and $B$, we also refer to $S$ as a {\em thick} surface. 
The {\em Heegaard genus} $g(M,\gamma)$ of the sutured manifold $(M,\gamma)$ is the minimum genus of its Heegaard surfaces.

\begin{remark}
Traditionally, the term {\em Heegaard splitting} is used for a decomposition of a manifold $M$ along a connected surface $S$ into either two handlebodies when $\bdry M = \emptyset$ or two connected compression bodies without vertical boundary when $\bdry M \neq \emptyset$.  We include such splittings among {\em linear Heegaard splittings}.  Ultimately, we opt to use the term {\em Heegaard splitting} to encompass a much larger class of splittings.
\end{remark}

\begin{remark}
Let us restrict attention here to sutured manifolds $(M,\gamma)$ without spherical components of $R(\gamma)$.
In \cite{juhasz}, Juh\'asz introduces {\em sutured diagrams} which are Heegaard diagrams for sutured manifolds $(M,\gamma)$.  Implicitly, these define Heegaard splittings for such sutured manifolds.   Conversely, given a Heegaard splitting $(A,B;S)$ of a sutured manifold $(M,\gamma)$, a minimal choice of handle structure on $A$ and $B$ induces a sutured diagram on $S$.  We leave the details of this to the reader as we are not concerning ourselves with sutured diagrams in this article.
\end{remark}

\begin{figure}
    \centering
    \includegraphics[width=4cm]{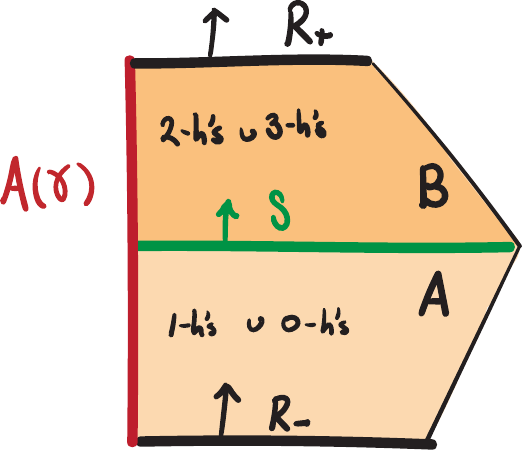}
    \caption{Heegaard splitting of $(M,\gamma)$ into compression bodies $A$ and $B$.}
    \label{fig:HSsuturedman}
\end{figure}

Every sutured manifold $M$ without toroidal sutures admits a Heegaard splitting. 
The {\em handle number} of a Heegaard splitting $(M,S) = (A,B;S)$ of a sutured manifold $M$ is $h(M,S) = h(A)+h(B)$.
The {\em (linear) handle number} 
$h(M,\gamma, 0)$ of the sutured manifold $M$ is the minimum of $h(M,S)$ among all linear Heegaard splittings of $(M,\gamma)$.
Observe that this is the minimum number of $0$--, $1$--, $2$--, and $3$--handles needed to construct $M$ from a collar of $R_-$.  

In the case of a product sutured manifold $M = F\times [-1,1]$ where $R_{\pm}(M) = F \times \pm1$, the Heegaard surface $F \times 0$ gives a Heegaard splitting of $M$ for which the compression bodies $A$ and $B$ are trivial. Thus its linear handle number is zero.  Indeed, the linear handle number of a sutured manifold is $0$ exactly when the sutured manifold is a  product.

While we often consider situations in which $M$ is connected, note that these definitions all apply when $M$ is a disconnected sutured manifold.  Indeed, a Heegaard splitting of $M$ restricts to a Heegaard splitting on each connected component and its handle number is the sum of the handle numbers of its connected components.
Furthermore, note that the Heegaard surface restricts to a connected Heegaard surface for each connected component of $M$.

\begin{remark} \label{rem:godadefn}
Goda defines the (linear) handle number of a connected sutured manifold $M$ in the situation where the surfaces $R_+$ and $R_-$ are non-empty and homeomorphic (and have no sphere components). Then for any Heegaard splitting $(A,B;S)$ of $M$, $h(A)$ is just a count of $1$--handles, and we have that $h(A) = h(B)$.  So Goda defines $h(M,\gamma)$ to be the minimum of $h(A)$ over all Heegaard splittings \cite{Goda-circlevaluedmorsetheory}.  Our definition takes $h(M,\gamma,0)$ to be twice that.
\end{remark}

\subsection{Circular Heegaard splittings } 
\label{sec:chs}

Let $(M,\gamma)$ be a sutured manifold, possibly with toroidal sutures.  A {\em circular Heegaard splitting} of $M$, or for simplicity just a {\em Heegaard splitting} of $M$, is a decomposition of $M$ along a pair of disjoint properly embedded oriented surfaces $F$ and $S$ so that
\begin{itemize}
    \item $F$ is a decomposing surface for $M$ 
    \item the sutured manifold decomposition of $M$ along $F$ yields a  sutured manifold $M \cut F$ without toroidal sutures, and
    \item $S$ is a Heegaard surface for $M \cut F$.
\end{itemize}
This last condition implies that $\bdry S$ is contained in the union $\gamma \cup \nbhd(\bdry F)$ of the sutures $\gamma$ and a collar neighborhood of $\bdry F$ in $\bdry M$ so that $\bdry S$ is the union of the cores of the annular sutures of $M \cut F$.
Furthermore, as $M\cut F$ is potentially disconnected, the surface $S$ restricts to a Heegaard surface on each component of $M \cut F$.
We may denote the Heegaard splitting as the triple $(M,F,S)$ or the quadruple $(A,B;F,S)$ where $(A,B;S)$ is a linear Heegaard splitting of $M\cut F$.
We say $S$ and $F$ are the {\em thick} and {\em thin} surfaces of the splitting as they are respectively the thick and thin boundaries of both $A$ and $B$, aside from $R_+$ and $R_-$.

\begin{remark}
\label{rem:genHeegspl}
Our definition of {\em Heegaard splitting} encompasses linear Heegaard splittings as defined above as well as the splittings often called ``generalized (linear) Heegaard splittings'' or ``generalized circular Heegaard splittings''.  
Indeed a linear Heegaard splitting is just a Heegaard splitting with $F=\emptyset$.
\end{remark}

Observe that in a  Heegaard splitting $(M,F,S)$, we have $[F,\bdry F] = [S, \bdry S] \in H_2(M, \bdry M)$.  Furthermore, for each toroidal suture $T$ of $\bdry M$, each $[\bdry F]$ and $[\bdry S]$ must restrict to the same non-zero class in $H_1(T)$.   
Let us say a class $\xi \in H_2(M, \bdry M)$ is {\em regular} (with respect to the sutured manifold structure) if any surface $F$ representing $\xi$ meets each toroidal suture.

Note that if $(M,\gamma)$ has no toroidal sutures, then every class is regular.  However, if $(M,\gamma)$ has toroidal sutures, then a regular class is necessarily non-zero.
For the exterior of a knot in $S^3$ whose boundary is a toroidal suture, then every non-zero class is regular.  
For the exterior of a two-component link in $S^3$ whose boundary is a pair of toroidal sutures, then every non-zero class is regular unless the linking number is zero; in that case each component bounds a Seifert surface disjoint from the other component.

The {\em handle number} of a (circular) Heegaard splitting $(M,F,S) = (A,B;F,S)$ is $h(M,F,S) = h(A)+h(B)$. 
For a sutured manifold $(M,\gamma)$ and decomposing surface $F$ that meets all toroidal sutures, 
the {\em handle number} $h(M,F)$ is  the minimum handle number $h(M,F,S)$ among Heegaard splittings of $(M,\gamma)$ with thin level $F$. (In other words, $h(M,F)$ is the linear handle number of the sutured manifold $M\cut F$ obtained by decomposing along $F$.)
For a sutured manifold $(M,\gamma)$ and regular class $\xi \in H_2(M, \bdry M)$,
the {\em handle number} $h(M,\gamma,\xi)$ is the minimum handle number $h(M,F,S)$ among Heegaard splittings of $(M,\gamma)$ (or equivalently, the minimum handle number $h(M,F)$ among decomposing surfaces) in which $F$ represents $\xi$, i.e.\ in which $[F,\bdry F] = \xi$.
We say a surface $F$ representing a regular class $\xi$ {\em realizes} $h(M,\gamma,\xi)$ if there is a surface $S$ in $M$ such that $h(M,\gamma,\xi) = h(M,F,S)$.
 
Observe that $0$ is a regular class for a sutured manifold $(M,\gamma)$ without toroidal sutures.  Hence the handle number for the trivial class is the linear handle number $h(M,\gamma,0)$.  This is in accordance with  Remark~\ref{rem:genHeegspl} and Lemma~\ref{lem:amalg} below.

\subsection{Product decompositions and handle numbers}
Consider a sutured manifold $(M,\gamma)$. A {\em product disk} is a properly embedded disk whose boundary is a union of an arc in $R_+(\gamma)$, an arc in $R_-(\gamma)$, and two spanning arcs of the annular sutures.  A {\em product annulus} is a properly embedded annulus with one boundary component in $R_+(\gamma)$ and the other in $R_-(\gamma)$.  A {\em product decomposition} is a sutured manifold decomposition where the decomposing surface is a union of product disks and product annuli.

The following lemma is well known in the theory of sutured manifold decompositions.  We include it without proof.
\begin{lemma}
\label{lem:commutativityoftransversesumdecomp}
Suppose $E$ and $F$ are transversally intersecting decomposing surfaces for $(M,\gamma)$ 
such that in any component of $A(\gamma)$ all arcs of $\bdry E$ and $\bdry F$ are coherently oriented.
Let $\delta$ be the collection of product disks and product annuli induced from the double curve sum.

Then we have the following diagram of sutured manifold decompositions.

\[\begin{tikzcd}
	&& {(M_E,\gamma_E)} \\
	{(M,\gamma)} && {(M_{E \dcsum F},\gamma_{E \dcsum F})} && {(M',\gamma')} \\
	&& {(M_F,\gamma_F)}
	\arrow["E", squiggly, from=2-1, to=1-3]
	\arrow["F"', squiggly, from=2-1, to=3-3]
	\arrow["{E \dcsum F}", squiggly, from=2-1, to=2-3]  
	\arrow["{F \cut E}", squiggly, from=1-3, to=2-5]
	\arrow["\delta", squiggly, from=2-3, to=2-5]
	\arrow["{E\cut F}"', squiggly, from=3-3, to=2-5]
\end{tikzcd}\]
\qed
\end{lemma}

\begin{remark}
The condition that all arcs of $\bdry E$ and $\bdry F$ are coherently oriented in any component of $A(\gamma)$ will be called condition (C1) in Section~\ref{sec:conditioned}.
\end{remark}

\begin{lemma}\label{lem:productdecompgivesnondecreasinghandlenumber}
Suppose $(M,\gamma)$ has no toroidal sutures, and $(M',\gamma')$ is obtained from $(M,\gamma)$ by a product decomposition. Equivalently, suppose $(M,\gamma)$ is obtained from $(M',\gamma')$ by gluing a pair of spanning rectangles in $A(\gamma')$ or a pair of components of $A(\gamma')$.
Then $h(M',\gamma',0) \geq h(M,\gamma, 0)$.
\end{lemma}

\begin{proof}
We state the proof for a decomposition along a product disk.  The case of a product annulus carries through similarly.

Let $S'$ be a Heegaard surface for $(M',\gamma')$ that realizes $h(M',\gamma', 0)$, splitting $(M',\gamma')$ into compression bodies $A'$ and $B'$.
Let $\delta_\pm$ be two spanning rectangles in $A(\gamma')$ so that identifying them produces $(M,\gamma)$.  Observe that $S'$ may be taken to intersect each of $\delta_\pm$ as the core arc of the annular sutures does.  The identification of $\delta_\pm$ may be done to also identify these arcs so that $S'$ becomes a surface $S$ in the sutured manifold $(M,\gamma)$. Likewise, the identification of $\delta_\pm$ restricts to an identification of the rectangles $A' \cap \delta_\pm$ of $A'$ and also of the rectangles $B' \cap \delta_\pm$ of $B'$. Consequently these identifications yield sutured manifolds $A$ and $B$ with $\bdry_+ A = S = \bdry_- B$ that contain product disks, and the decompositions along these product disks yield the compression bodies $A'$ and $B'$.  By \cite[Lemma 2.4]{Goda-HSforSMandMuraSum}, $A$ and $B$ are compression bodies with $h(A)=h(A')$ and $h(B)=h(B')$.   Hence $S$ is a Heegaard surface for $(M,\gamma)$ with $h(M,S) = h(A)+h(B) = h(A')+h(B') = h(M',\gamma',0)$.  Since $h(M,\gamma,0) \leq h(M,S)$, we have $h(M,\gamma,0) \leq h(M',\gamma',0)$.
\end{proof}

\subsection{Handle structures, handle counts, and handle indices}\label{sec:handlecounts}

Given a handle structure on a manifold $X$, let $\mathcal{H}^k(X)$ be the collection of its $k$--handles, and let $h^k(X) = | \mathcal{H}^k(X) |$ be the number of $k$--handles.

The {\em handle number} of a sutured $3$--manifold $M$ with a handle structure is $h(M) = h^0(M)+h^1(M)+h^2(M)+h^3(M)$.
The {\em handle index} of a sutured $3$--manifold $M$ with a handle structure is $j(M) = (h^1(M)-h^0(M))+(h^2(M)-h^3(M))$.

Observe that for a {\em minimal} handle structure on a connected compression body $A$ and a {\em minimal} handle structure on a connected dualized compression body $B$ we have 
\begin{equation*}
h^0(A) =
        \begin{cases} 
        0  &\mbox{ if } \bdry_-A\neq \emptyset\\   
        1 &\mbox{ if } \bdry_-A= \emptyset
        \end{cases}
\qquad \mbox{ and } \qquad
h^3(B) =
        \begin{cases} 
        0  &\mbox{ if } \bdry_+B\neq \emptyset\\   
        1 &\mbox{ if } \bdry_+B= \emptyset
        \end{cases}
\end{equation*}
so that 
\begin{equation*}
h(A) =
        \begin{cases} 
        j(A)  &\mbox{ if } A \mbox{ is not a handlebody}\\   
        j(A)+2 &\mbox{ if } A \mbox{ is a handlebody}
        \end{cases}
\qquad \mbox{ and } \qquad
h(B) =
        \begin{cases} 
        j(B)  &\mbox{ if } B \mbox{ is not a handlebody}\\   
        j(B)+2  &\mbox{ if } B \mbox{ is a handlebody}
        \end{cases}
\end{equation*}

Let $(M,\gamma)$ be a connected sutured manifold without toroidal sutures. 
Define $d=d(M,\gamma)$ to be $0$, $1$, or $2$ depending on whether neither, just one, or both of $R_+$ and $R_-$ are empty.

Then for any linear Heegaard splitting of $(M,\gamma)$ realizing $h(M,\gamma,0)$ with handle structure that is minimal on each compression body, observe that $d = h^0(M) + h^3(M)$.

More generally, let $(M,F,S) = (F,S;A,B)$ be a circular Heegaard splitting of the sutured manifold $(M,\gamma)$.  Since $h(M,F,S) = h(A)+h(B) = h^0(A) + h^1(A) + h^2(B) + h^3(B)$, the number $d(M,F,S) = h^0(A)+h^3(B)$ counts the number of handlebodies in the splitting.  Then if $(M,F,S)$ is a splitting realizing $h(M,\gamma,\xi)$,  we may define $d(M,\gamma, \xi) = d(M,F,S)$.   Observe that $d(M,\gamma,0)=d(M,\gamma)$ as above while $d(M,\gamma,\xi) =0$ if $\xi \neq 0$.


\section{Some operations on Heegaard splittings}

\subsection{Amalgamations}

Consider a circular Heegaard splitting $(M, F, S)$ of $M$ into a  compression body $A$ and dual compression body $B$.  Suppose a collection of components $A_1$ of $A$ and $B_0$ of $B$ meet along a collection of components $F_1$ of $F$ (so that $F_1 = \bdry_- A_1 \cap \bdry_+ B_0$).  If $\bdry_+ A_1 \cap \bdry_- B_0 = \emptyset$ then we can {\em amalgamate} $(M, F, S)$ along $F_1$ to produce a new circular Heegaard splitting $(M',F',S')$ with $F' = F - F_1$.

Since we assume $\bdry_+ A_1 \cap \bdry_- B_0 = \emptyset$, the definition of amalgamation along a surface given in \cite[Definition 12]{Baker-MN} extends to this case where $A_1$ and $B_0$ may be disconnected. (As noted in Remark~\ref{rem:genHeegspl}, our definition of {\em circular Heegaard splitting} encompasses splittings that are often called {\em generalized Heegaard splittings} or {\em generalized circular splittings}.)  See also \cite[\S2]{SchultensClassificationOfHeegSpl} and \cite[\S3]{LackenbyAnAlgorithmtodetermine} for other treatments of the amalgamation operation.  

\begin{lemma}\label{lem:amalg}
If $(M',F',S')$ is an amalgamation of $(M,F,S)$, then $h(M',F',S') \leq h(M,F,S)$.
\end{lemma}

\begin{proof}
From the definition of amalgamation (see \cite[Definition 12]{Baker-MN}), the initial handle structure on $(M,F,S)$ extends to a handle structure on $(M',F',S')$ with the same number of handles of each dimension.  Hence we have $h(M',F',S') \leq h(M,F,S)$.
(Note that equality does not necessarily hold since it is possible that induced handle structure on $(M',F',S')$ now allows for some cancellations.)
\end{proof}

\begin{remark}
\label{rem:amalgamationhomology}
If $(M',F',S')$ is an amalgamation of $(M,F,S)$ along the surface $F_1 \subset F$, then observe that $[F'] + [F_1] = [F]$.  Hence $[F']=[F]$ only if $[F_1] = 0$.
\end{remark}

\subsection{Realization by homologically essential surfaces}
A properly embedded surface $F$ in a $3$--manifold $M$ is {\em homologically essential} if no non-empty subcollection of its components is null-homologous in $H_2(M,\bdry M)$.

\begin{lemma}\label{lem:homologicallyessential}
Suppose $\xi$ is a regular class.
If $\xi\neq0$ and $F$ realizes $h(M,\gamma,\xi)$, then there is a homologically essential subsurface $\overline{F} \subset F$ that also realizes $h(M,\gamma,\xi)$.

If $\xi=0$, then $h(M,\gamma,0)$ is realized by a Heegaard splitting $(M,S)$.
\end{lemma}

\begin{proof}
Suppose $F$ realizes $h(M,\gamma,\xi)$ with the Heegaard splitting $(M,F,S)$.
Further suppose $F$ contains a non-empty collection of components that is null-homologous.  Then $F$ contains a minimal such collection $F_1$,  for which no proper non-empty subcollection is null-homologous.  In particular, $F_1$ is the boundary of a submanifold of $M$.
Then we have collections of components $A_1$ of $A$ and $B_0$ of $B$ for which $F_1 = \bdry_-A_1 \cap \bdry_+ B_0$.  They have their corresponding collections of components $B_1$ of $B$ and $A_0$ of $A$ where $\bdry_+ A_i = \bdry_- B_i = S_i$ is a collection of components of $S$ for $i=0,1$. 
Suppose $S_0 \cap S_1 \neq \emptyset$ so that $S_0$ and $S_1$ share a component $S^*$.  Then there are components $A^*$ of $A$ and $B^*$ of $B$ so that $S^* = \bdry_+ A^* = \bdry_- B^*$ and $F_1$ has a component in each $\bdry_- A^*$ and $\bdry_+ B^*$.  Thus there is a path in $A^* \cup_{S^*} B^*$ from one side of $F_1$ to the other side of $F_1$.  However that is contrary to $F_1$ bounding a submanifold of $M$.
Thus $S_0 \cap S_1 = \emptyset$ and we may amalgamate $(M,F,S)$ along $F_1$ to obtain the splitting $(M',F',S')$ where $F' = F-F_1$.  By Lemma~\ref{lem:amalg}, $h(M',F',S') \leq h(M,F,S)$.  Since $F_1$ is null-homologous, $[F']=[F]$ as noted in Remark~\ref{rem:amalgamationhomology}.  Since  $F$ realizes $h(M,\gamma,\xi)$, the previous inequality is an equality and $F'$ realizes $h(M,\gamma,\xi)$ too.  This may be repeated until we are left with a Heegaard splitting $(M, \overline{F}, \overline{S})$ that also realizes $h(M,\gamma,\xi)$ where $\overline{F} \subset F$ is a homologically essential subsurface, possibly the empty set.
\end{proof}

\subsection{Conditioned splittings}\label{sec:conditioned}

Given a sutured manifold $(M,\gamma)$, a {\em conditioned} surface $F$ is a decomposing surface of $(M,\gamma)$ such that
\begin{enumerate}
    \item[(C1)] all arcs of $\bdry F$ in any component of $A(\gamma)$ are coherently oriented, and
    \item[(C2)] no non-empty collection of simple closed curves $\bdry F \cap R(\gamma)$ is trivial in $H_1(\bdry M, \gamma)$.
\end{enumerate}
We further say $F$ is {\em well-conditioned} if
\begin{enumerate}
    \item[(C3)] $F$ is homologically essential.
\end{enumerate}
A Heegaard splitting $(M,F,S)$ of $(M,\gamma)$ is {\em conditioned} if $F$ is conditioned, and it is {\em well-conditioned} if $F$ is well-conditioned.

\begin{lemma}\label{lem:realizeC1}
Suppose $\xi$ is a regular class for $(M,\gamma)$. Then there is a surface $F$ realizing $h(M,\gamma,\xi)$ that satisfies condition (C1).
\end{lemma}

\begin{proof}
Let $F_0$ be a decomposing surface for $(M,\gamma)$ that represents $\xi$.
Suppose two adjacent arcs of $F_0 \cap A(\gamma)$ are anti-parallel and joined by an arc $\alpha$ of the core suture.
Then let $F_1$ be the surface obtained by a boundary tubing of $F_0$ along $\alpha$ that is disjoint from the rectangle of $A(\gamma)$ defined by alpha and these two arcs.  Then $F_1$ has a boundary compressing disk $\delta_\alpha$ that crosses $\alpha$ once and boundary compresses $F_1$ back into $F_0$. 
(Note that $F_1$ also represents $\xi$.)    In the result $(M_1,\gamma_1)$ of the sutured manifold decomposition of $(M,\gamma)$ along $F_1$, $\delta_\alpha$ is a product disk.  Furthermore, the decomposition along $\delta_\alpha$ results in the sutured manifold $(M_0, \gamma_0)$ obtained from decomposing $(M,\gamma)$ along $F_0$.  Consequently, $(M_1,\gamma_1)$ can be regarded as being obtained from $(M_0,\gamma_0)$
 by regluing a product disk corresponding to $\delta_\alpha$.  
See Figure~\ref{fig:realizeC1}.  By Lemma~\ref{lem:productdecompgivesnondecreasinghandlenumber}, $h(M_1, \gamma_1,0) \leq h(M_0, \gamma_0,0)$.  In particular, there is a Heegaard surface $S_1$ of $(M_1, \gamma_1)$  such that $h(M,F_1,S_1) \leq h(M,F_0, S_0)$ for any Heegaard surface $S_0$ of $(M_0, \gamma_0)$.

\begin{figure}
    \centering
    \includegraphics[width=.9\textwidth]{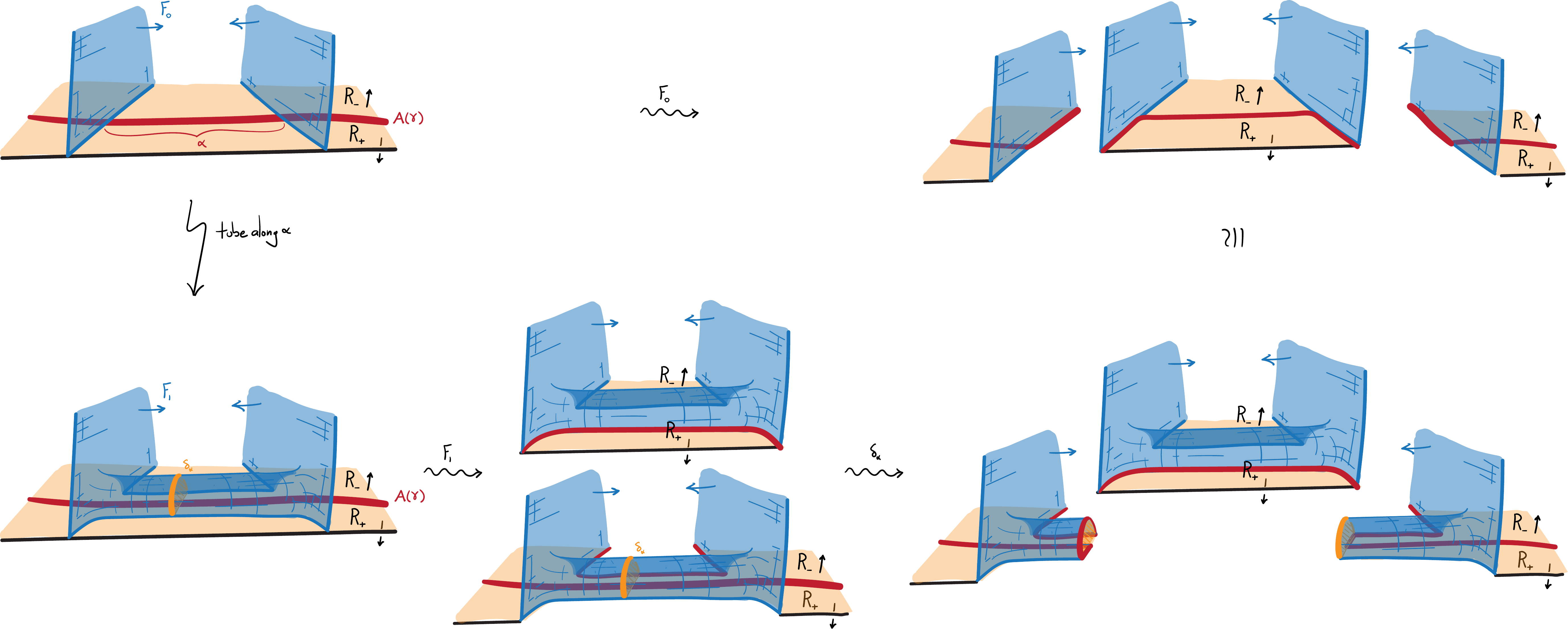}
    \caption{Left: A boundary tubing of $F_0$ along $\alpha$ produces $F_1$ with its boundary compressing disk $\delta_\alpha$.  Center: Upon decomposing $(M,\gamma)$ along $F_1$ to obtain $(M_1, \gamma_1)$, the disk $\delta_\alpha$ becomes a product disk.  Right: The result of decomposing $(M_1, \gamma_1)$ along the product disk $\delta_\alpha$ is homeomorphic to $(M_0,\gamma_0)$ obtained by decomposing $(M,\gamma)$ along $F_0$. }
    \label{fig:realizeC1}
\end{figure}

Hence it follows that $h(M,\gamma, \xi)$ is realized by a surface $F$ whose arcs of intersection with $A(\gamma)$ are coherently oriented in any component of $A(\gamma)$.
\end{proof}

\begin{lemma}\label{lem:dcsumheegsplitting}
Let $F$ be a decomposing surface for a sutured manifold $(M,\gamma)$ that meets every toroidal suture and satisfies property (C1).
Let $R_0$ be a component of $R=R(\gamma)$, let $F_{R_0}$ be the double curve sum of $F$ with a push-off of $R_0$, and let $F_{R_0}'$ be the result of discarding any components of $F_{R_0}$ that are $\bdry$--parallel into $R_0$.
Then
\begin{enumerate}
    \item $F_{R_0}'$ is homologous to $F$,
    \item $|F_{R_0}' \cap R| \leq |F \cap R|$,
    \item $|F_{R_0}'| \leq |F|$, and
    \item $h(M,F_{R_0}') \leq h(M,F)$.
\end{enumerate}
Furthermore, 
\begin{enumerate}
    \item[(2')] $|F_{R_0}' \cap R| < |F \cap R|$ if some component of $R_0 \cut F$ gives the null-homology in $H_1(M,\gamma)$ of a non-empty collection of simple closed curves of $\bdry F \cap R_0$, and
    \item[(3')] $|F_{R_0}'| < |F|$ if some component of $R_0 \cut F$ meets two components of $F_\pm$ with consistent orientations.
\end{enumerate}
\end{lemma}

\begin{proof}
Fix a collar $C = R_0 \times [0,1]$ in $M$ where $R_0 \times \{0\} = R_0$ and $\bdry R_0 \times [0,1]$ is contained in $A(\gamma)$ so that $F \cap C$ is a collar in $F$ of $F \cap R_0$.
In particular, $F \cap C$ may be identified with $(\bdry F \cap R_0) \times [0,1]$.  
Then $F_{R_0}$ is the double curve sum  $R_0\times \{1/2\} \dcsum F$.  
 Since $F$ satisfies property (C1), so does $F_{R_0}$.
Furthermore, since $R_0\times \{1/2\}$ is null-homologous in $H_2(M,\gamma)$, $F_{R_0}$ is homologous to $F$.
Then let $F_{R_0}'$ be $F_{R_0}$ with any component that is $\bdry$--parallel into $R_0$ discarded, so $F_{R_0}'$ is homologous to $F$ too.  Observe that  $F \cut C = F_{R_0} \cut C = F_{R_0}' \cut C$.  See Figure~\ref{fig:peeling}.

\begin{figure}
    \centering
    \includegraphics[width=.9\textwidth]{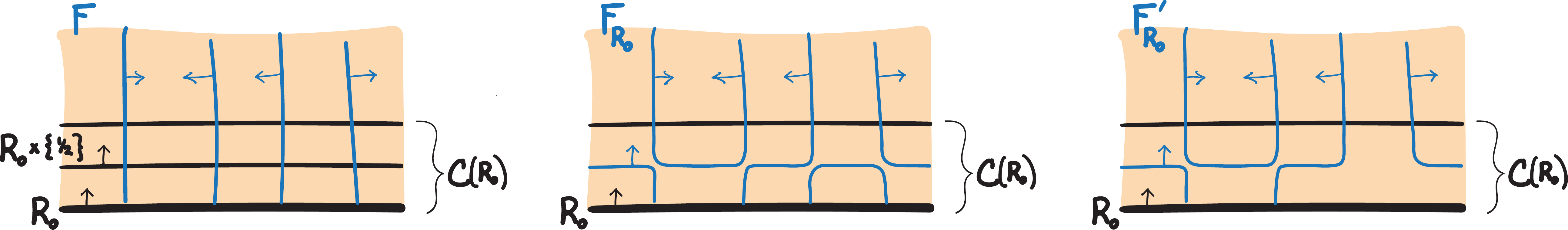}
    \caption{The surfaces $F$, $F_{R_0}$, and $F_{R_0}'$ (assuming the lower-left piece of $F_{R_0}$ is not parallel into $R_0$).}
    \label{fig:peeling}
\end{figure}

Since $F \cap C$ is a collar of $\bdry F \cap R_0$,  $|F| = |F \cut C| = |F_{R_0} \cut C|$.  Since the components of $F_{R_0}$ that meet $\bdry F$ but not $F \cut C$ are necessarily $\bdry$--parallel into $R_0$ since they are disjoint from $R_0 \times \{1\}$. (Otherwise they would meet $F \cut C$).  These are the components that are discarded to form $F_{R_0}'$. Consequently,  $|F_{R_0}' \cap R| \leq |F \cap R|$.  
Now suppose there is a component $Q$ of $R_0 \cut F$ such that $\bdry Q$ is a union of components $\mathcal{C}$ of $\bdry F$ (with induced boundary orientation) and possibly some components of $\bdry R_0$ (with any orientation).  Then $[\mathcal{C}] = \bdry [Q]$ in $H_1(\bdry M, \gamma)$.  Furthermore, one component of the double curve sum $F_{R_0}$ is formed from the union of $Q \times \{1/2\}$ and the annular collars $\mathcal{C} \times [0,1/2]$ in $F$. This component is $\bdry$--parallel to $Q \subset R_0$ and is discarded when forming $F_{R_0}'$.  Thus $\bdry F_{R_0}'$ is a proper subcollection of $\bdry F_{R_0}$, and hence $|F_{R_0}' \cap R| < |F \cap R|$.

As $F_{R_0}'$ is the union of the components of $F_{R_0}$ that meet $F \cut C$, it has at most as many components as $F$.  Hence $|F_{R_0}'| \leq |F|$. However, further observe that if some component $Q$ of $R_0 \cut F$ meets two components of $F_\pm$ with consistent orientations (that is, if say $R_0 \subset R_-$, then this component $Q$ should meet two components of $F_-$), then in the double curve sum $Q \times \{1/2\}$ is joined to $(Q \cap F) \times [1/2,1]$ which then joins to the corresponding components of $F \cut C$.  Hence $F_{R_0}'$ will have fewer components than $F_{R_0}' \cut C$.  Since $|F| = |F \cut C|$ and $F \cut C = F_{R_0}' \cut C$, we have that $|F_{R_0}'| < |F|$.

Finally, to show that $h(M,F_{R_0}') \leq h(M,F)$, observe that $h(M,F_{R_0})=h(M,F_{R_0}')$ immediately follows from the 
definition of $F_{R_0}'$.   So we aim to show that $h(M,F_{R_0}) \leq h(M,F)$.
Let $(M_{F_{R_0}}, \gamma_{F_{R_0}})$ be the sutured manifold decomposition of $(M,\gamma)$ along $F_{R_0}$.  
Let $R_{1/2} = R_0 \times \{1/2\}$. Lemma \ref{lem:commutativityoftransversesumdecomp} implies we have the following diagram:
\[\begin{tikzcd}
	&& {(M_{R_{1/2}},\gamma_{R_{1/2}})} \\
	{(M,\gamma)} && {(M_{F_{R_0}},\gamma_{F_{R_0}})} && {(M',\gamma')} \\
	&& {(M_F,\gamma_F)}
	\arrow["R_{1/2}", squiggly, from=2-1, to=1-3]
	\arrow["F"', squiggly, from=2-1, to=3-3]
	\arrow["{F_{R_0}}", squiggly, from=2-1, to=2-3] 
	\arrow["{F \cut R_{1/2}}", squiggly, from=1-3, to=2-5]
	\arrow["\delta", squiggly, from=2-3, to=2-5]
	\arrow["{R_{1/2}\cut F}"', squiggly, from=3-3, to=2-5]
\end{tikzcd}\]
where $\delta$ is a set of product disks and annuli in $(M_{F_{R_0}}, \gamma_{F_{R_0}})$ arising from the double curve sum of $F$ and $R_{1/2}$ that produces $F_{R_0}$. We observe that:
\begin{itemize}
    \item $(M_{R_{1/2}},\gamma_{R_{1/2}})$ is homeomorphic to $(M, \gamma) \cup (R_0 \times [0, 1/2], \gamma \cap (R_0 \times [0, 1/2]))$,
    \item $F \cut R_{1/2} \cong F \cup (\bdry F \cap R_0) \times [0, 1/2]$, and
    \item $R_{1/2}\cut F$ is parallel to $R_0 \cut F$ in $(M_F, \gamma_F)$.
\end{itemize}
Hence $(M', \gamma')$ is the union of $(M_F, \gamma_F)$ and a product sutured manifold. So $h(M', \gamma',0) = h(M_F, \gamma_F, 0) = h(M,F)$.
Thus we have 
\[ h(M,F) = h(M_F, \gamma_F, 0) = h(M',\gamma',0) \geq h(M_{F_{R_0}},\gamma_{F_{R_0}},0) = h(M,F_{R_0}) \]
where the inequality follows from 
Lemma~\ref{lem:productdecompgivesnondecreasinghandlenumber}.
\end{proof}

\begin{theorem}\label{thm:wellconditioned}
For any sutured manifold $(M,\gamma)$ and regular class $\xi \in H_2(M,\bdry M)$ there is a well-conditioned Heegaard splitting that realizes $h(M,\gamma,\xi)$.
\end{theorem}

\begin{proof}
By Lemma~\ref{lem:realizeC1}, there is a surface realizing $h(M,\gamma,\xi)$ that satisfies condition (C1).
Assume that $F_0$ has been chosen among all such surfaces $F$ to minimize the complexity $(|F|, |F \cap R|)$ lexicographically. 

Suppose that there is a non-empty collection of simple closed curves $\bdry F_0 \cap R(\gamma)$ is trivial in $H_1(\bdry M, \gamma)$. Then there is a component $R_0$ of $R$ that contains such a collection.  In particular, there is an oriented collection $\mathcal{C} = C_1 \cup \dots \cup C_n$ of components of $\bdry F_0$ in $R_0$ so that $\mathcal{C}$ bounds (with perhaps some components of $\bdry R_0$) a connected subsurface $Q$ of $R_0$ that contains no other such collection. If the interior of $Q$ is disjoint from $\bdry F_0$, then Lemma~\ref{lem:dcsumheegsplitting} produces a surface with smaller complexity.  Hence $\bdry F_0$ must intersect the interior of $Q$.  In particular there is a component $C_0$ of $\bdry F_0$ (possibly an arc)  in the interior of $Q$ and an arc $\alpha$ embedded in $Q$ running from $\mathcal{C}$ to the same side of $C_0$ that is otherwise disjoint from $\bdry F_0$.  
(Due to condition (C1), all arcs of $Q \cap \bdry F_0$ must meet any component of $\bdry R_0$ in $\bdry Q$ coherently.
So if there is an arc $\alpha'$ embedded in $Q$ running from $\mathcal{C}$ to an arc of $Q \cap \bdry F_0$, then there is an arc $\alpha$ as described.)
A boundary tubing of $F_0$ along $\alpha$ produces a surface $F_\alpha$ so that $M \cut F_\alpha$ decomposed along a product disk $\delta_\alpha$ (the boundary compression of the tube along $\alpha$) gives the same as $M \cut F_0$, see Figure~\ref{fig:boundarytubing}.    So by Lemma~\ref{lem:productdecompgivesnondecreasinghandlenumber} $F_\alpha$ must also realize the handle number.  But $F_\alpha$ meets $R$ in fewer components, contrary to assumption.   Thus $F_0$ must satisfy condition (C2).

\begin{figure}
    \centering
    \includegraphics[width=.7\textwidth]{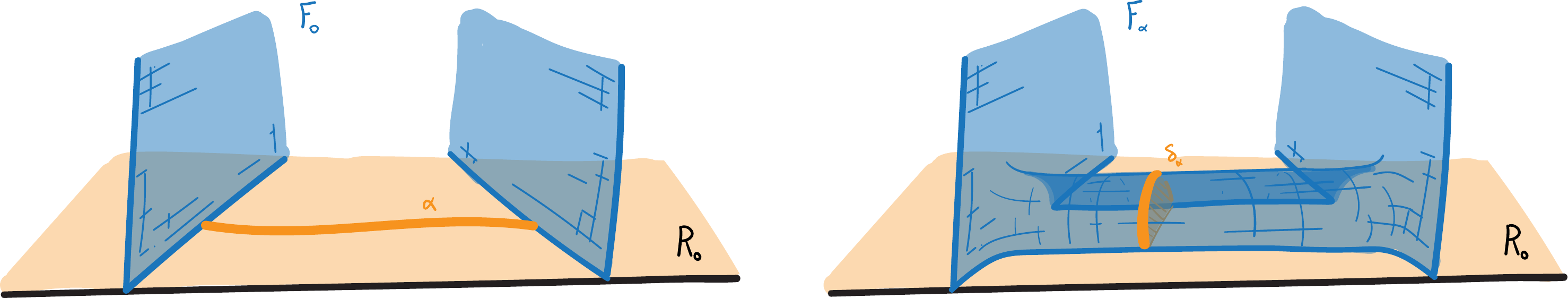}
    \caption{A boundary tubing of a thin surface $F_0$ (left) along an arc $\alpha$ produces a surface $F_\alpha$ (right).  The induced boundary compressing disk $\delta_\alpha$ of $F_\alpha$ is a product disk for the decomposition along $F_\alpha$.}
    \label{fig:boundarytubing}
\end{figure}

 Since any collection of components of a conditioned surface continues to be conditioned,  Lemma~\ref{lem:homologicallyessential} shows that we may discard any homologically inessential components of $F_0$ to obtain a surface $\overline{F}_0$ satisfying condition (C3) that also realizes the handle number.  
  
 In particular, $\overline{F}_0$ is well-conditioned and we have a well-conditioned splitting $(M,\overline{F}_0,\overline{S}_0)$ that realizes $h(M,\gamma,\xi)$.
\end{proof}

\subsection{Tubings of splittings}
Let $(M,F,S)$ be a  Heegaard splitting of $(M,\gamma)$ with a non-trivial compression body $A_i$ containing a $1$--handle $\calH^1_i$ whose feet are in $F$.
Let $B_j$ and $B_k$ be the dual compression bodies that meet the feet of $\calH^1_i$. (Note that we may have $B_j = B_k$.) Let $A_j$ and $A_k$ be the corresponding compression bodies where $\bdry_+ A_j = \bdry_- B_j = S_j$ and $\bdry_+ A_k = \bdry_- B_k = S_k$ are each components of $S$.  Form the compression body $A_i' = A_i \cut \calH^1_i$. In $B_j \cup B_k \cup \calH^1_i$, form a $1$--handle $\calH^1_{jk}$ with feet in $S_j \cup S_k$ from a (narrower) regular neighborhood of the core of $\calH^1_i$ and two vertical arcs in $B_j$ and $B_k$ joining the endpoints of the core to $S_j \cup S_k$.   
Now $B_{jk}' = (B_j \cup B_k \cup \calH^1_i) \cut \calH^1_{jk}$ is a dual compression body and $A_{jk}' = A_j \cup A_k \cup \calH^1_{jk}$ is a compression body that meet along the new thick surface $S'_{jk}$ that is the tubing of $S_j \cup S_k$ along $\bdry \calH^1_{jk}$.   
Also observe that the compression body $A_i' = A_i \cut \calH^1_i$ and the dual compression body $B_{jk}' = (B_j \cup B_k \cup \calH^1_i) \cut \calH^1_{jk}$ meet along the thin surface $F'_{jk}$ that is the tubing of $\bdry_-A_j$ and $\bdry_-A_k$ along $\bdry \calH^1_{jk}$.
This forms a new Heegaard splitting $(M,F',S')$ that we say is the {\em  tubing}  of $(M,F,S)$ along the $1$--handle $\calH^1_i$. 
See Figure \ref{fig:handlextension}.

\begin{figure}
    \centering
    \includegraphics[width=.7\textwidth]{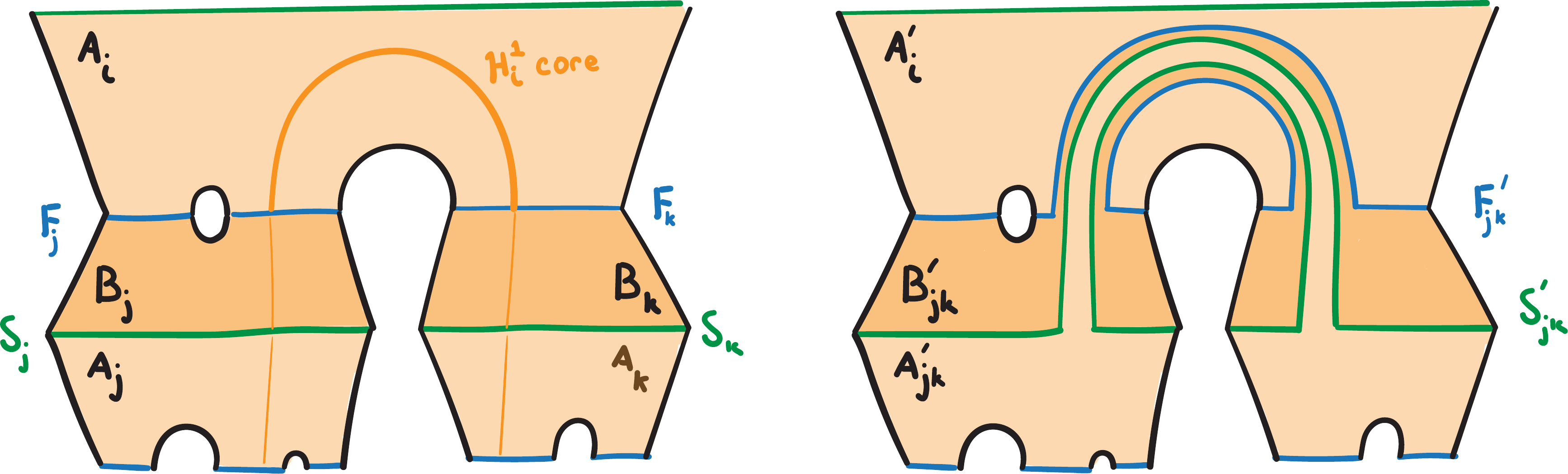}
    \caption{Left: A Heegaard splitting  $(M,F,S)$ with a $1$--handle $\calH^1_i$ in $A_i$ whose feet are in $F$. Right: The new Heegaard splitting $(M,F',S')$  is the {\em  tubing}  of $(M,F,S)$ along the $1$--handle $\calH^1_i$.}
    \label{fig:handlextension}
\end{figure}

\begin{lemma}\label{lem:handleext}
If $(M,F',S')$ is a tubing of $(M,F,S)$, then $[F']=[F]$, $|F'| \leq |F|$, and $h(M,F',S') \leq h(M,F,S)$.  In particular, if the tubing occurs along a handle with feet in distinct components of $F$, then $|F'|=|F|-1$.
\end{lemma}

\begin{proof}
From the definition of tubing along a handle, the initial handle structure of $(M,F,S)$ extends to a handle structure on $(M',F',S')$ with the same number of handles of each dimension.  Hence $h(M,F',S') \leq h(M,F,S)$.    Also, from the definition, if the tubing is done along a handle with feet in distinct components of $F$, then it follows that $|F'|=|F|-1$.
\end{proof}

\begin{remark}
Observe that the operation of tubing a splitting preserves the property of being well-conditioned.  That is, if $(M,F',S')$ is a tubing of a well-conditioned splitting $(M,F,S)$, then $(M,F',S')$ is well-conditioned too.
\end{remark}


\section{Restructuring sutures}\label{sec:restructure}
To get around the issue of non-zero classes that may fail to be regular for link exteriors where the boundary tori are all toroidal sutures, we may instead choose the boundary tori to all belong to $R_\pm$. 
This latter sutured manifold structure allows the circular handle number to be defined for all non-zero classes by ensuring they are all regular.
Moreover, for regular classes of the former ``all toroidal suture'' sutured manifold structure, the resulting circular handle number is the same as with this latter ``all $R_\pm$'' structure.  More generally, we have the following.

\begin{theorem}\label{thm:reassign}
Consider sutured manifolds structures $(M,\gamma)$ and $(M,\gamma^*)$ which agree on $\bdry M \cut T(\gamma)$ and for which $T(\gamma^*)=\emptyset$.

If $\xi \in H_2(M,\bdry M)$ is a regular class for $(M,\gamma)$, 
then either 
\begin{itemize}
    \item $h(M,\gamma,\xi) = h(M,\gamma^*,\xi)$,
    \item 
    there is a component of $T(\gamma)$ in which $\bdry \xi$ and $A(\gamma^*)$ have the same slope, or
    \item there is a component of $T(\gamma)$ in which some component of $R_\pm(\gamma^*)$ is a disk.
\end{itemize}
In particular, if $T(\gamma) \subset R_\pm(\gamma^*)$ then 
$h(M,\gamma,\xi) = h(M,\gamma^*,\xi)$ for every regular class of $(M,\gamma)$.
\end{theorem}

 \begin{proof}
 Let us assume the last conclusion of the theorem does not occur; that is, in no component of $T(\gamma)$ is some component of $R_\pm(\gamma^*)$ a disk.
 Then this theorem is primarily a consequence of the following observation.  
 Let $\calT$ be a collection of torus boundary components of a $3$--manifold $M$.  Let $F$ be a properly embedded, oriented surface in $M$ whose boundary meets each component of $\calT$ in a non-empty collection of coherently oriented essential simple closed curves.  Then if $(M,\gamma)$ and $(M,\gamma^*)$ are two sutured manifold structures on $M$ that differ only on $\calT$ such that $F$ is a decomposing surface for each, the sutured manifolds that result from their decompositions along $F$ are homeomorphic unless one has annular sutures on a component of $\calT$ of the same slope as the curves of $\bdry F$.

 Now, with the hypotheses of the theorem, suppose $\xi \in H_2(M,\bdry M)$ is a regular class for $(M,\gamma)$.  
 Then for every torus $T$ in $T(\gamma)$, the restriction of 
 $\bdry \xi$ to $H_1(T)$ is not $0$.
   So assume the following:
 \begin{quote} 
 ($\dagger$) 
 The slopes of $\bdry \xi$ and $A(\gamma^*)$ are distinct in each component of $T(\gamma)$.
 \end{quote}
 Hence if $F$ is a decomposing surface for $(M,\gamma)$ representing $\xi$, it meets each component of $T(\gamma)$ in a non-empty collection of coherently oriented simple closed curves.  By the above observation with our assumption ($\dagger$), the results of decomposing each $(M,\gamma)$ and $(M,\gamma^*)$ along $F$ are homeomorphic to the same sutured manifold $(M',\gamma')$.  
 Then for any splitting $(M,F,S)$ of $(M,\gamma)$ realizing $h(M,\gamma,\xi)$ there is a splitting $(M,F,S^*)$ of $(M,\gamma^*)$ with $h(M,F,S^*) = h(M,F,S) = h(M,\gamma,\xi)$.  Hence $h(M,\gamma^*,\xi) \leq h(M,\gamma,\xi)$.
 
On the other hand, suppose $(M,F^*,S^*)$ is a splitting of $(M,\gamma^*)$ realizing $h(M,\gamma^*,\xi)$.  
By Theorem~\ref{thm:wellconditioned}, we may assume $(M,F^*,S^*)$ is well-conditioned.  In particular, this means $\bdry F^*$ is coherently oriented on each component of $T(\gamma)$.
Then, again by the above observation that begins this proof with our assumption ($\dagger$), the results of decomposing each $(M,\gamma)$ and $(M,\gamma^*)$ along $F^*$ are homeomorphic.  Hence there is a splitting $(M,F^*,S)$ of $(M,\gamma)$ with $h(M,F^*,S)=h(M,F^*,S^*)=h(M,\gamma^*,\xi)$, and so $h((M,\gamma,\xi) \leq h(M,\gamma^*, \xi)$.
 
Together this implies that, under assumption ($\dagger$), we have $h((M,\gamma,\xi) = h(M,\gamma^*, \xi)$.
Furthermore, if $T(\gamma) \subset R_\pm(\gamma^*)$, then $A(\gamma^*) \cap T(\gamma) = \emptyset$ so that assumption ($\dagger$) holds for all regular classes $\xi$ of $(M,\gamma)$.
 \end{proof}

A similar proof gives the following.

\begin{cor}
Consider sutured manifolds structures $(M,\gamma_1)$ and $(M,\gamma_2)$ which only disagree on a collection of tori of $\bdry M$ in which $A(\gamma_i)$ are essential annuli.  

If $\xi \in H_2(M,\bdry M)$ is a regular class for $(M,\gamma_1)$ and $(M,\gamma_2)$, 
then either 
\begin{itemize}
    \item $h(M,\gamma_1,\xi) = h(M,\gamma_2,\xi)$ or
    \item 
    there is a torus component of $\bdry M$ in which $\bdry \xi$ and $A(\gamma_i)$ have the same slope for either $i=1,2$. \qed
\end{itemize}

\end{cor}


\section{An upper bound for sutured manifolds without toroidal sutures}
\subsection{Decompositions of compression bodies}

Given a connected compact oriented surface $\Sigma$, the product $P = \Sigma \times I$ is a trivial compression body which has the structure of a product sutured manifold $(P,\gamma)$ where the sutures $\gamma = \bdry \Sigma \times \{0\}$ have annular neighborhoods $A(\gamma) = \bdry \Sigma \times I$.  Here we take $I = [-1,1]$.      Furthermore we have the oriented surfaces  $R_\pm = \Sigma \times \{\pm1\}$ for which the normal orientation is the $I$ direction.
Being a product sutured manifold, $(P, \gamma)$ is taut, e.g.\ \cite[Definition 2.2]{Scharlemann_SM}.

\begin{theorem}\label{thm:cuttoproduct}
Let $\Sigma$ be a connected compact oriented surface other than a sphere.
Let $P = \Sigma \times I$ be the induced product sutured manifold.  
Let $F_0$ be a conditioned surface in $P$.
Then there is a conditioned surface $F$ in $P$ where
\begin{enumerate}
    \item $\bdry F \cut A(\gamma) = \bdry F_0 \cut A(\gamma)$,
    \item $\bdry F$  is homologous to  $[\bdry F_0] + m [\gamma]$ in $H_1(\bdry P)$ for some non-negative integer $m$,
    \item $[F,\bdry F] = [F_0,\bdry F_0]$ in $H_2(P,\bdry P)$, 
    \item $F$ is taut, and
    \item the sutured manifold decomposition of $(P, \gamma)$ along $F$ is again a product.
\end{enumerate}
In particular, if $\Sigma$ is closed so that $A(\gamma) = \emptyset$, then $\bdry F = \bdry F_0$.
\end{theorem}

\begin{proof}
As $P = \Sigma \times I$ and $\Sigma \not \cong S^2$, it is irreducible and $\Sigma \times\{0\}$ is a taut surface.
Setting $y = [F_0, \bdry F_0]\in H_2(P, \bdry P)$ so that $\bdry y = [\bdry F_0]$, \cite[Theorem 2.5]{Scharlemann_SM} 
implies there exists a taut surface $F$ in $P$ so that $\bdry F \cut A(\gamma) = \bdry F_0 \cut A(\gamma)$ and $[F, \bdry F] = y + m [ \Sigma \times \{0\}, \bdry \Sigma \times \{0\}]$ in $H_2(P, \bdry P)$ for some non-negative integer $m$.
In particular, this implies $[\bdry F] = [\bdry F_0] + m [\gamma]$ in $H_1(\bdry P)$ since $\gamma = \bdry \Sigma \times \{0\}$. 
We may further assume that any oppositely oriented components of $\bdry F$ in $A(\gamma)$ or $T(\gamma)$ have been capped off with annuli so that $F$ is a conditioned surface and any closed separating components have been discarded (as done in the proof of \cite[Theorem 2.6]{Scharlemann_SM}).

Let $(P',\gamma')$ be the sutured manifold obtained from decomposing $(P,\gamma)$ along $F$.
Since $\Sigma \times \{\pm1\}$ is a pair of similarly oriented parallel copies of $\Sigma \times \{0\}$, \cite[Theorem 2.5(iii)]{Scharlemann_SM} implies that the double curve sum $F \dcsum \Sigma \times \{\pm 1\}$ is taut in $P$, and hence also taut in $P' = P \cut F$ (again, see also the proof of \cite[Theorem 2.6]{Scharlemann_SM}). In particular, $R(\gamma')$ is taut in $P'$.
Next, note that any sphere in $P'$ bounds a $3$--ball in $P$ (because $P$ is irreducible) that is disjoint from $F$ since $F$ has no closed separating components;
 thus $P'$ is irreducible.
Therefore, by \cite[Definition 2.2]{Scharlemann_SM}, this shows that the sutured manifold $(P',\gamma')$ 
is a taut sutured manifold:  we have that both $P'$ is irreducible and $R(\gamma')$ is taut. 
Thus the sutured manifold decomposition $(P,\gamma) \overset{F}{\leadsto} (P',\gamma')$ is taut.   Now \cite[Lemma 2.4]{Gabai_detectingfibredlinks}  (see also \cite[\S4]{Gabai-FT3M-I}) implies that the resulting sutured manifold is also a product manifold.  Thus $F$ satisfies the conclusion.
\end{proof}

\begin{lemma}
\label{lem:compressionbodymeridianandverticaldecomp}
Let $C$ be a connected compression body with handle number $h(C)$.  
For a handle structure on $C$ realizing $h(C)$, let $\calD$ be a collection of coherently oriented copies of meridional disks that are co-cores of $1$--handles. 
Let $\calA$ be a  collection of  disks and  annuli that are vertical with respect to this handle structure.   
Then the decomposition of $C$ along $\calD \cup \calA$ is a compression body $C'$ with handle number $h(C') = h(C)$ if $\bdry_-C \neq \emptyset$ and $h(C') = h(C)-2$ if $C$ is a handlebody. 
In either case, the handle index is preserved: $j(C') = j(C)$.
\end{lemma}

\begin{proof}
First decompose $C$ along $\calA$.  Since $\calA$ is a collection of vertical surfaces, $C \cut \calA$ is again a compression body.  Moreover, viewing $C\cut \calA$ as a submanifold of $C$, it inherits the handle structure of $C$. In particular, $h(C \cut \calA) = h(C)$ and $\calD$ continues to be a collection of coherently oriented copies of meridional disks that are co-cores of $1$-handles of this inherited handle structure.  So we continue the argument with the compression body $C\cut \calA$. Hence we may as well proceed assuming $\calA = \emptyset$. 

One observes that the region between a pair of parallel and adjacent disks in $\calD$ yields a product ball after the decomposition.  Since products may be components of compression bodies and they contribute nothing towards the handle number, we further reduce to assuming $\calD$ has no parallel disks. Say $d= |\calD|$. 
Then $C$ has a handle structure in which the disks of $\calD$ are co-cores of some of the handles.  Decomposing $C$ along $\calD$ therefore cuts these handles open. Forgetting about any sutured manifold structure, $C'=C \cut \calD$ is a compression body with handle number  $h(C)-d$.
Each disk of $\calD$ leaves two impressions in $\bdry C'$.  So, keeping track of the sutures induced from the decomposition of $C$ along $\calD$, we have a sutured manifold structure on $C'$ where $R_-(C')$ is the disjoint union of $R_-(C)$ and a  collection of $d$ disks in $\bdry C'$, one for each disk in $\calD$.  $R_+(C')$ is then just $\bdry C' - \nbhd(R_-(C'))$; indeed $R_+(C')$ is obtained by adding a copy of each disk in $\calD$, with the appropriate orientation, to what is left of $R_+(C)$.  Thus, as a sutured manifold, $C'$ is a compression body obtained by joining a thickening of these $d$ disks of to a thickening of $R_-(C)$ each with a $1$-handle, if $R_-(C) \neq \emptyset$.  Thus $h(C') = h(C)$.    
Figure~\ref{fig:CBHBdecomposition}(Left) illustrates this situation.

\begin{figure}
    \centering
    \includegraphics[width=.7\textwidth]{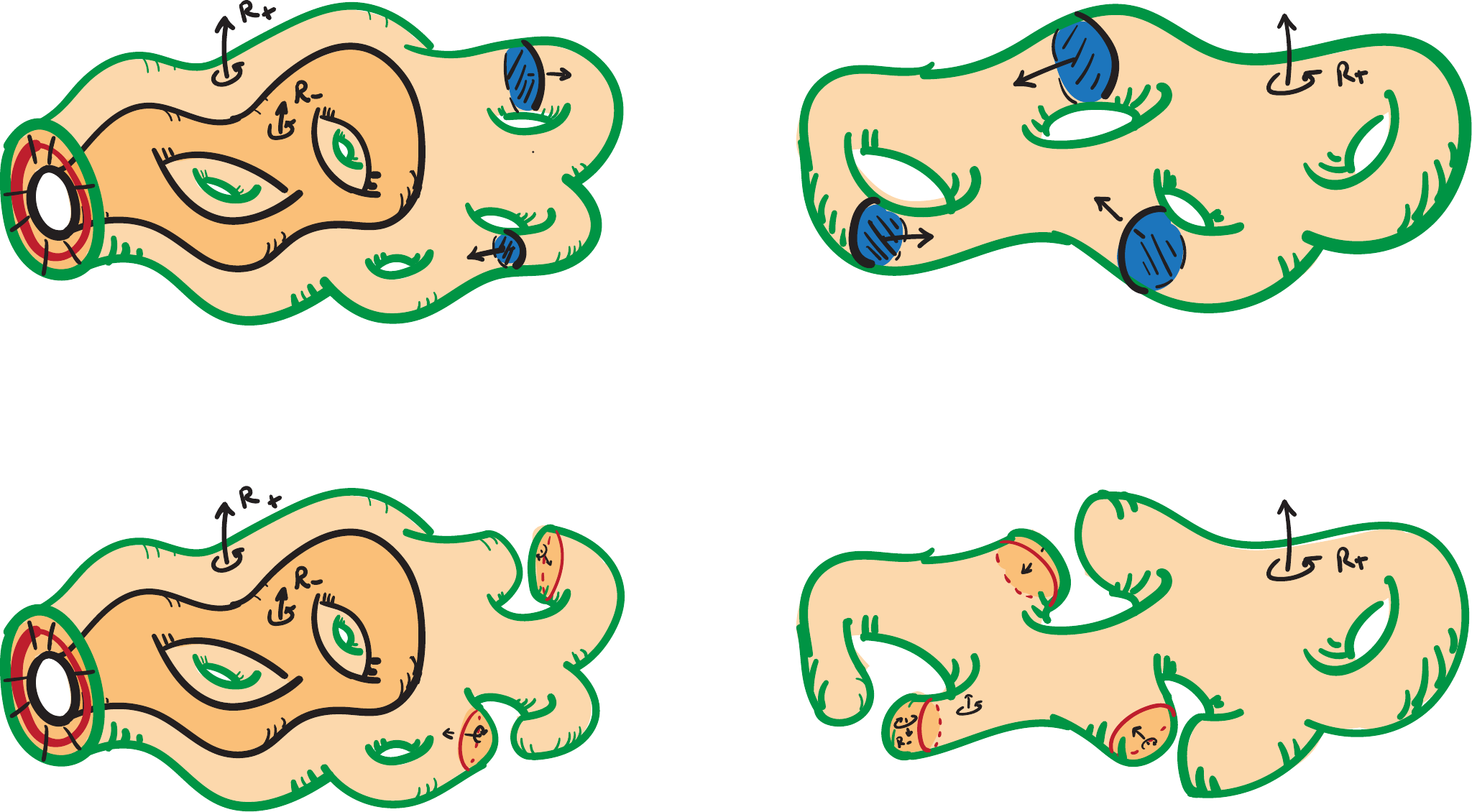}
    \caption{Left: A compression body cut along a collection $\calD$.  Right: A handlebody cut along another collection $\calD$}
    \label{fig:CBHBdecomposition}
\end{figure}

If $R_-(C) = \emptyset$ (so that $C$ is a handlebody) then $R_-(C')$ is a collection of $d$ disks in $\bdry C'$, one for each disk in $\calD$.  Then, as a sutured manifold,
$C'$ is obtained by first joining a thickening of these $d$ disks to the $0$--handle of $C$ with $d$ $1$--handles
and then adding the remaining $h(C)-d-1$ $1$--handles that were disjoint from the disks of $\calD$. Now, with a choice of one of the $1$--handles from a thickened disk of $R_-(C')$ to the $0$--handle, the other $1$--handles may be slid off the $0$--handle and down this $1$--handle to the thickened disk.  Finally, the $0$--handle and this $1$--handle may be cancelled.  Thus $h(C') = h(C)-2$. Figure~\ref{fig:CBHBdecomposition}(Right) illustrates this situation.
\end{proof}

The following theorem is a generalization of \cite[Lemma 2.4]{Goda-HSforSMandMuraSum} which says that a sutured manifold is a compression body if and only if it is obtained after decomposing  a compression body along a product disk, and that the handle number is preserved under such decomposition.

\begin{theorem}\label{thm:splitcompressionbody}
Let $C$ be a connected compression body with sutures $\gamma$ and handle number $h(C)$. 
Let $F_0$ be a 
conditioned surface.
Then there is a conditioned  
surface $F$ in $C$ such that $\bdry F \cut A(\gamma)= \bdry F_0 \cut A(\gamma)$ and $[F, \bdry F] = [F_0, \bdry F_0] + m[\bdry_+ C, \bdry(\bdry_+ C)]$ in $H_2(C, \bdry C)$ for some integer $m>0$ with the following property. The decomposition of $C$ along $F$ is a  compression body $C'$ with handle number  $h(C') = h(C)$ if $\bdry_-C \neq \emptyset$ and $h(C')=h(C)-2$ if $C$ is a handlebody.  In either case the handle index is preserved: $j(C') = j(C)$.
\end{theorem}

\begin{proof}
A {\em spine} of a compression body $C$ is the union of $\bdry_- C$ and an embedded graph $\tau$ so that $C\cut (\bdry_-C \cup \tau) \cong \bdry_+C \times[0,1]$.
Choose a spine $\bdry_-C \cup \tau$ of $C$ induced by a handle structure realizing $h(C)$, and consider a thin closed neighborhood $N = \bdry_-C \times [-1,-1+\epsilon] \cup N_\tau$ where $N_\tau$ is a closed regular neighborhood of $\tau$. Note that $N$ is a compression body homeomorphic to $C$ and  $C \cut N$ is the product sutured manifold $P= \bdry_+C \times I$.  Let $\Sigma$ be the surface $\bdry_+ N = \bdry_- P$ so that $C = N \cup_\Sigma P$.

Observe that $F_0$ may be isotoped in $C$ to meet this spine so that $F_0 \cap \bdry_-C$ is properly embedded collection of arcs and loops in $\bdry_-C$ and $F_0$ intersects $\tau$ in a discrete set of interior points.  Thus we may assume $F_0$ intersects $N$ in a collection of meridional disks, vertical disks, and vertical annuli.  Adjust $F_0$ by tubing along subarcs of $\tau$ if necessary to ensure that within any parallelism class these meridional disks are coherently oriented, and update $F_0$ to be the resulting surface. This preserves that $F_0$ is conditioned.  Since $F_0$ is conditioned, both $F_0 \cap N$ and $F_0 \cap P$ are conditioned surfaces in $N$ and $P$.
Set $F_N = F_0 \cap N$.  
Next apply Theorem~\ref{thm:cuttoproduct} to the surface $F_0 \cut N = F_0 \cap P$ in the product $P$ to obtain a conditioned surface $F_P$.
Then together these form the conditioned surface $F = F_P \cup F_N$ which is properly embedded in $C$.

We claim that $F$ is now our desired surface.
Since $F \cap N = F_N = F_0 \cap N$ while $F \cap P = F_P$ is obtained from Theorem~\ref{thm:cuttoproduct}, we constructed $F$ so that
 $\bdry F \cut A(\gamma) = \bdry F_0 \cut A(\gamma)$ and and $[F, \bdry F] = [F_0, \bdry F_0] + m[\bdry_+ C, \bdry(\bdry_+ C)]$ in $H_2(C, \bdry C)$ for some integer $m>0$.  Furthermore, it gives the desired decomposition.
 
 Indeed, the result $C'$ of the decomposition of $C$ along $F$ is the union of the compression body $N' = N \cut F_N$ from Lemma~\ref{lem:compressionbodymeridianandverticaldecomp} and the product $P' = P \cut F_P$ from Theorem~\ref{thm:cuttoproduct} along the impressions of surface $\Sigma' = \Sigma \cut F$ in $\bdry N'$ and $\bdry P'$.
More specifically, since $F$ transversally intersects $\Sigma$, the decomposition of $C = N \cup_{\Sigma} P$ as a sutured manifold along $F$ may be viewed as follows.  Figure~\ref{fig:decompandglue} gives a schematic overview. 
\begin{itemize}
    \item First we decompose $C$ along $\Sigma$ into the product $P$ and compression body $N$.   Under this decomposition,  $\Sigma$ leaves  impressions $\Sigma_+ = R_+(N)$ and $\Sigma_- = R_-(P)$ and splits $F$ into $F_N \subset N$ and $F_P \subset P$.  Note that $R_-(N) = R_-(C)$ and $R_+(P) = R_+(C)$.
    \item Next we decompose $N$ along $F_N$ to obtain the compression body $N'$ and $P$ along $F_P$ to obtain the product $P'$.  Identifying $\Sigma'$ as $\Sigma \cut F$,  we observe that $R_+(N') = \Sigma'_+ \cup (F_N)_+$ and $R_-(P') = \Sigma'_- \cup (F_P)_-$.   Furthermore $R_-(N') = (R_-(N) \cut F_N) \cup (F_N)_-$ and $R_+(P') = (R_+(P) \cut F_P) \cup (F_P)_+$.
    \item Finally we glue $P'$ to $N'$, identifying $\Sigma'_+ \subset R_+(N')$ with $\Sigma'_- \subset R_-(P')$.  This forms a sutured manifold $C'$ and a properly embedded surface $\Sigma'$ arising from the identification of $\Sigma'_+$ with $\Sigma'_-$.  
    An interval of $\bdry \Sigma'$ belongs to the sutures of $C'$ only if it is in both or neither of $\bdry \Sigma'_+$ and $\bdry \Sigma'_-$; see Figure~\ref{fig:gluing}.  (By construction, every point of $\bdry \Sigma'$ is in at least one of $\bdry \Sigma'_+$ or $\bdry \Sigma'_-$.)  Thus we have $R_-(C') = R_-(N') \cup (F_P)_-$ and $R_+(C') = R_+(P') \cup (F_N)_+$.  
    \item Therefore, because $F_- = (F_N)_- \cup (F_P)_-$, and $F_+ = (F_N)_+ \cup (F_P)_+$, we have 
    $R_-(N) \cut F_N = R_-(C) \cut F$ and  $R_+(P) \cut F_P = R_+(C) \cut F$.  Then it follows that 
    \[R_-(C')  = R_-(N') \cup (F_P)_- = (R_-(N) \cut F_N) \cup (F_N)_- \cup (F_P)_-  = (R_-(C) \cut F) \cup F_- \]
    and 
    \[R_+(C') = R_+(P') \cup (F_N)_+ = (R_+(P) \cut F_P) \cup (F_N)_+ \cup (F_P)_+ = (R_+(C) \cut F) \cup F_+.\]
    So we see that $C'$ is the sutured manifold obtained from the decomposition of $C$ along $F$.
\end{itemize}

\begin{figure}
    \centering
    \includegraphics[width=.8\textwidth]{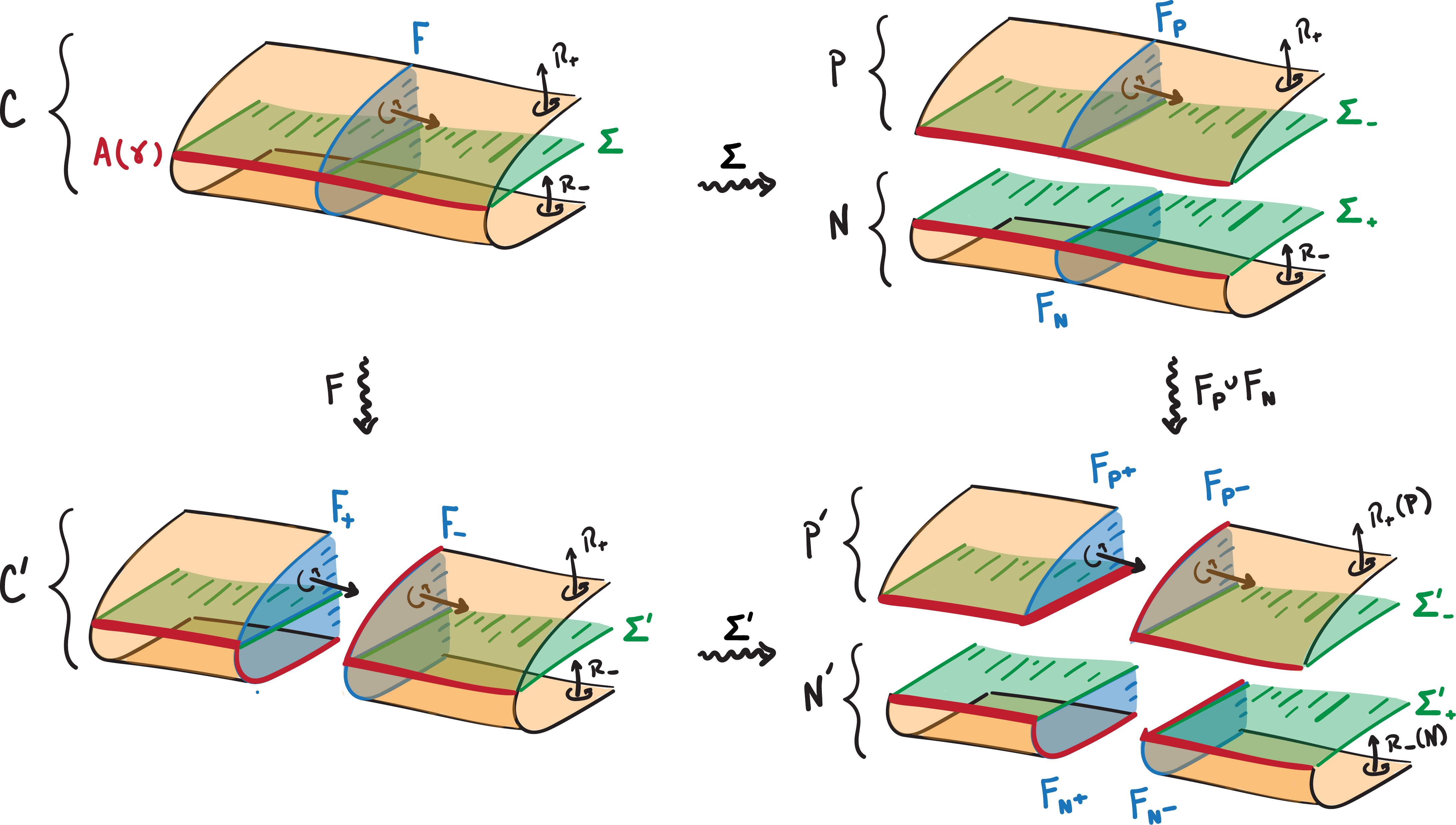}
    \caption{A schematic of how decompositions and gluings of sutured manifolds behave near the surfaces $\Sigma$ and $F$ in $C$.  In particular, the decomposition along $F$ may be recovered by a decomposition along $\Sigma$, followed by a decomposition along $F_P \cup F_N$, and gluing the resulting sutured manifolds together along  $\Sigma'_+$ and $\Sigma'_-$. }
    \label{fig:decompandglue}
\end{figure}

\begin{figure}
    \centering
    \includegraphics[height=4cm]{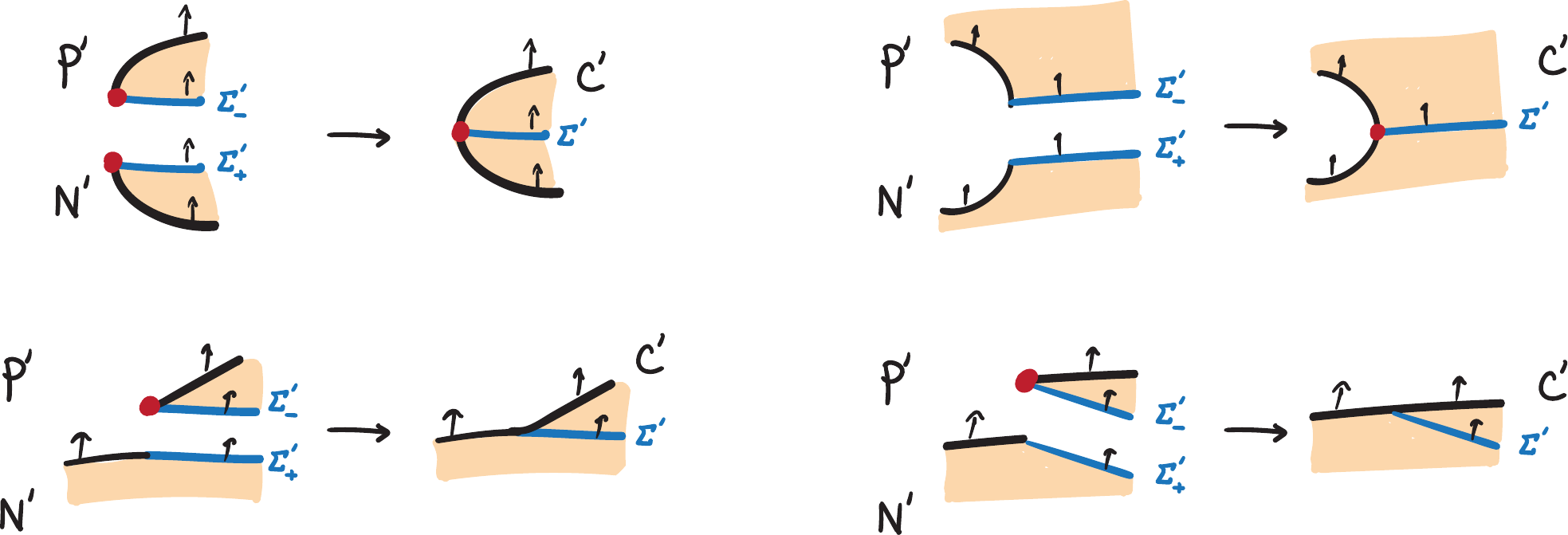}
    \caption{A schematic of how sutures behave under gluing a sutured manifold along subsurfaces $\Sigma'_- \subset R_-$ and $\Sigma'_+ \subset R_+$.}
    \label{fig:gluing}
\end{figure}

  \begin{figure}
     \centering
     \includegraphics[width=.95\textwidth]{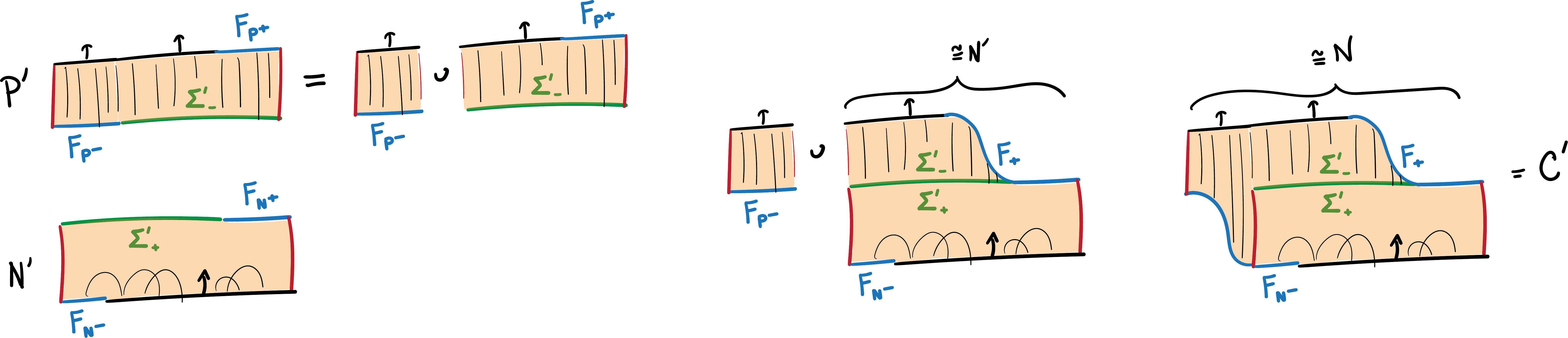}
     \caption{The compression body $C'$ is obtained by attaching the product over $\Sigma'_-$ in $P'$ to $R_+(N')$ (to get a sutured manifold homeomorphic to $N'$) followed by attaching the product over $(F_P)_-$  in $P'$ along their sutures.}
     \label{fig:descomp5}
 \end{figure}

Now we observe that $C'$ is indeed a compression body and we determine its handle number.
Since $P'$ is a product with $R_-(P')  = \Sigma'_- \cup (F_P)_-$, it splits along the vertical disks and annuli over the curves $\bdry \Sigma'_- \cap \bdry (F_P)_-$ as a product over $\Sigma'_-$ and a product over $(F_P)_-$.
The curves $\bdry \Sigma'_- \cap \bdry (F_P)_-$ which are not in the sutures of $P'$ however are identified with $\bdry \Sigma'_+ \cap (F_N)_-$ where they do belong to the sutures of $N'$.   Thus we may regard $C'$ as being obtained by attaching the product over $\Sigma'_-$ in $P'$ to $R_+(N')$ (to get a sutured manifold homeomorphic to $N'$) followed by attaching the product over $(F_P)_-$  in $P'$ along their sutures.  Hence $C'$ is sutured manifold homeomorphic to the sutured manifold $N'$ extended along its sutures by a product. See Figure~\ref{fig:descomp5}. Thus $h(C') = h(N')$. Moreover, since $N'$ is a compression body, so is $C'$.

As $h(N')$ is either $h(N)-2$ or $h(N)$ depending on whether or not $N$ is a handlebody by Lemma~\ref{lem:compressionbodymeridianandverticaldecomp}, we have our result about handle numbers.  The result for handle index immediately follows.
\end{proof}

\subsection{An Upper Bound}

\begin{theorem}\label{thm:lineartocircular}
Let $(M,\gamma)$ be a connected sutured manifold without toroidal sutures.  
For any non-zero class $\xi \in H_2(M,\bdry M)$, there is a surface $F$ representing $\xi$ so that $(M,\gamma)$ has a Heegaard splitting $(M,F,S)$ such that $h(M,F,S) = h(M,\gamma,0)-2d$.  
Consequently,
\[ h(M,\gamma,\xi) \leq h(M,\gamma,0)-2d.\]
Equivalently,
\[j(M,\gamma,\xi) \leq j(M,\gamma,0).\]
\end{theorem}

Recall that  $d=0$, $1$, or $2$ depending on whether neither, just one, or both of $R_+$ and $R_-$ are empty and that  $h(M,\gamma,0)$ is the linear handle number of $(M,\gamma)$.

\begin{proof}
Since $(M,\gamma)$ has no toroidal sutures, every $\xi \in H_2(M,\bdry M)$ is regular. 

Let $(M,\Sigma) = (A,B;\Sigma)$ be a Heegaard splitting of the sutured manifold $(M,\gamma)$. So the splitting has handle number $h(M,\Sigma)=h(A)+h(B)$.

By Theorem~\ref{thm:wellconditioned}, there is a conditioned surface $F_0$ representing $\xi$ that realizes $h(M,\gamma, \xi)$. By a small isotopy, one may assume it is transverse to $\Sigma$.

In particular, $\bdry F_0$ intersects $A(\gamma)$ in essential simple closed curves and coherently oriented spanning arcs.
Since $\Sigma$ is a Heegaard surface, $\bdry \Sigma$ is the union of the core curves of $A(\gamma)$.  So for any given component of $\bdry \Sigma$, its intersections with $\bdry F_0$ are all of the same sign.  Hence if a non-empty collection of components of $F_0 \cap \Sigma$ is trivial in $H_1(\Sigma, \bdry \Sigma)$, the conditioning of $F_0$ implies that the collection must be disjoint from the boundary.  
Thus we may further assume $F_0$ has been chosen so that no non-empty submanifold of $F_0 \cap \Sigma$ is trivial in $H_1(\Sigma, \bdry \Sigma)$. For example, one may modify an initial conditioned surface by taking its double curve sum with a collection of parallel copies of $\Sigma$.

For each $C=A,B$,   the restriction of $F_0$ to $C$ is a conditioned surface $F_0 \cap C$.
Then Theorem~\ref{thm:splitcompressionbody}  implies there is a conditioned surface $F_C$ in $C$ homologous to $F_0 \cap C$ with $\bdry (F_0 \cap C) \cap \Sigma = \bdry F_C \cap \Sigma$ so that the sutured manifold decomposition of $C$ along $F_C$ is a compression body $C'$ with $h(C') = h(C)$ or $h(C')=h(C)-2$ if $C$ is a handlebody.  Recall that a connected compression body is a handlebody exactly when $\bdry_-(C)=\emptyset$. 

So then let $F = F_A \cup F_B$.  Observe that by construction $F$ is homologous to $F_0$ so that it represents $\xi$, and moreover $F \cap \Sigma = F_0 \cap \Sigma$.  Let $S$ be the surface $F\dcsum \Sigma$ pushed off $F$.  Explicitly, we may take a closed collar neighborhood $N_F = F \times [-\epsilon, \epsilon]$ 
of $F$ so that it intersects $\Sigma$ in a collar of $F \cap \Sigma$.  Then $S = F_A \times \{-\epsilon\} \cup (\Sigma \cut N_F) \cup F_B \times \{+\epsilon\}$ is a properly embedded oriented surface disjoint from $F$.   We then claim that $(M, F, S)$ is a circular Heegaard splitting with handle number $h(A')+h(B')$.  Here $A'$ and $B'$ are the compression bodies obtained by decomposing $A$ and $B$ along $F_A$ and $F_B$.

Since $\bdry \Sigma$ is the collection of core curves of the annular sutures of $M$, we may regard the compression bodies $A'$ and $B'$ as being the components on either side of $\Sigma$ in $(M \cut \Sigma) \cut N_F$ or equivalently in $(M \cut N_F) \cut \Sigma$.  
Then as $F$ splits $N_F$ into $F \times [-\epsilon, 0]$ and $F \times [0, \epsilon]$,  we observe that  $S$ splits $M\cut F$ into the two manifolds $\overline{A'} = A' \cup F \times [0,\epsilon]$ and $\overline{B'} = B' \cup F \times [-\epsilon, 0]$.  See Figure~\ref{fig:schematic}.

Let us examine $\overline{A'}$.  Partitioning $F \times [0,\epsilon]$ along the vertical disks and annuli of its intersection with $\Sigma$, we have  $\overline{A'} = (A' \cup F_A \times [0,\epsilon]) \cup (F_B \times [0,\epsilon])$.  Since  $A' \cup F_A \times [0,\epsilon]$ is equivalent to $A'$, we see that $\overline{A'}$ decomposes along these vertical disks and annuli into $A'$ and the product $F_B \times [0,\epsilon]$.  Hence $\overline{A'}$ is a compression body with $h(\overline{A'}) = h(A')$.     A similar argument shows that $\overline{B'}$ is a compression body with $h(\overline{B'}) = h(B')$. 
Thus $(M,F,S)$ is a circular Heegaard splitting that divides $M$ into the compression bodies $\overline{A'}$ and $\overline{B'}$.  Hence $h(M,F,S) = h(\overline{A'})+h(\overline{B'})=h(A')+h(B')$.

By Theorem~\ref{thm:splitcompressionbody}, $h(A') = h(A)$ or $h(A)-2$ depending on whether or not $R_-(\gamma)$ is empty.  Similarly $h(B') = h(B)$ or $h(B)-2$ depending on whether or not $R_+(\gamma)$ is empty.  
Therefore $h(M,F,S) = h(A)+h(B)-2d = h(M,\Sigma)-2d$.    So $h(M,\gamma,\xi) \leq h(M,F,S) = h(M,\Sigma)-2d$.

Consequently if $(M,\Sigma)$ realizes the linear handle number $h(M,\gamma,0)$, then we have our claimed result that
\[ h(M,\gamma,\xi) \leq h(M,\gamma,0)-2d.\]
In terms of handle index, this becomes
\[j(M,\gamma,\xi) \leq j(M,\gamma,0).\]
\end{proof}

\begin{remark}

Note that in the proof above, even if $\Sigma$ is connected, both $\overline{A'}$ and $\overline{B'}$ may be disconnected when $F$ is separating in both $A$ and $B$.
 Furthermore $\overline{A'}$ and $\overline{B'}$ may have components that are products, i.e.\ trivial compression bodies.
\end{remark}

\begin{figure}
    \centering
     \includegraphics[width=\textwidth]{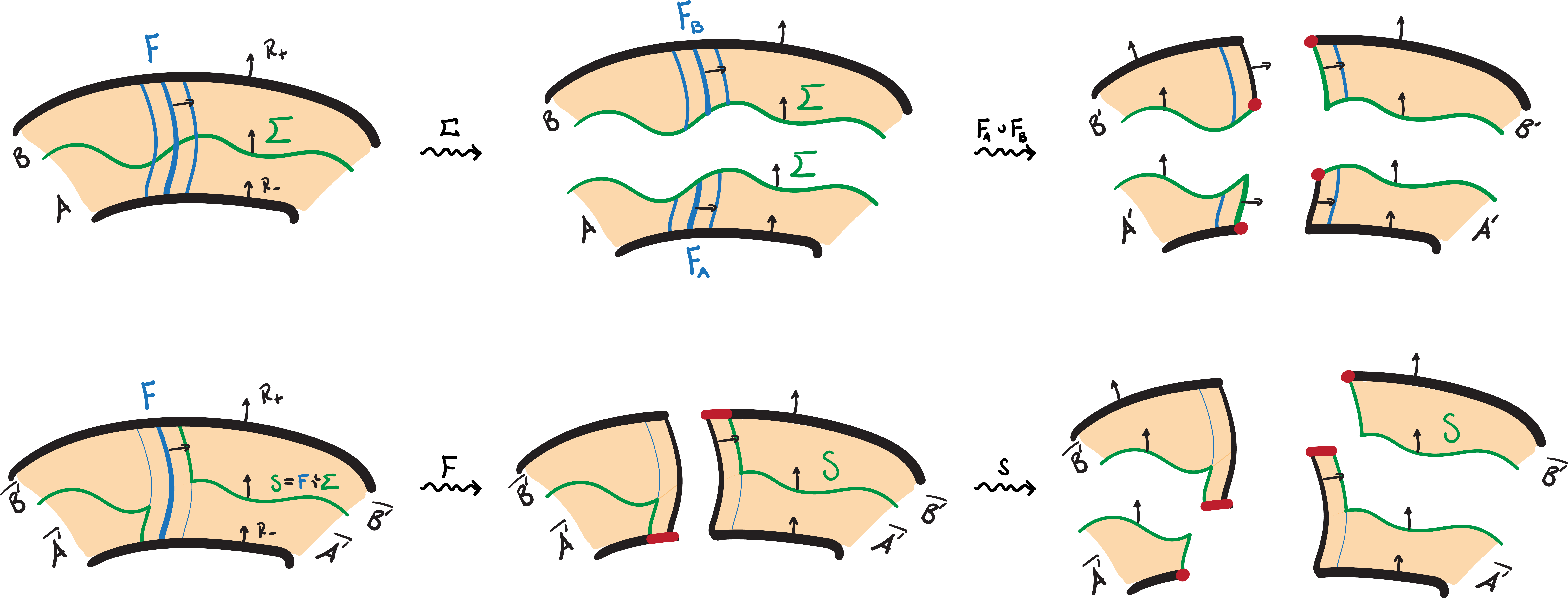}
    \caption{In the top row, we start with a Heegaard surface $\Sigma$ splitting $(M,\gamma)$ into compression bodies $A,B$. The surface $F$  restricted to each compression body $A,B$ decomposes it into a new compression body $A',B'$.  In the bottom row, $\Sigma$ is replaced with $S = F \dcsum \Sigma$.  Then $S$ is a Heegaard surface for the sutured manifold  $(M',\gamma')$ obtained by decomposing $(M,\gamma$) along $F$.  The surface $S$ splits $(M',\gamma')$ into compression bodies $\overline{A'},\overline{B'}$.}
    \label{fig:schematic}
\end{figure}

\section{A lower bound for sutured manifolds without toroidal sutures}

\subsection{Sutured Characteristic}
A {\em sutured surface} is a compact oriented surface $F$ with a disjoint pair of submanifolds $r_+$ and $r_-$ of $\bdry F$ such that  
\begin{itemize}
    \item the orientation of $r_+$ is consistent with the boundary orientation of $\bdry F$ while the orientation of $r_-$ is reversed,
    \item $\bdry F \cut (r_+ \cup r_-)$ is a collection of intervals and circles, collectively called the {\em sutures}, and
    \item each interval suture joins a component of $\bdry r_+$ to a component of $\bdry r_-$.
\end{itemize}
At times it is convenient to use $\bdry_+ F$ and $\bdry_- F$ to refer to the $1$--manifolds $r_+$ and $r_-$, $\bdry_v F$ (for {\em vertical} boundary) to refer to the interval sutures, and $\bdry_0 F$ for the circle sutures.

We now define a modified Euler characteristic for sutured surfaces.
When the sutured surface $F$ is connected, we define the {\em sutured characteristic} of $F$ to be
\[  \chi_{sut}(F) =  -\chi(F) + d(F) + \frac12 \#\bdry_v F + \#\bdry_0 F. \]
where $d(F) = 0$ if $\bdry_0 F \neq \emptyset$ and otherwise counts how many of $r_+$ and $r_-$ are empty.  That is, if $\bdry_0 F = \emptyset$ then $d(F) = 0$, $1$, or $2$ depending on whether neither, just one, or both  of $r_+$ and $r_-$ are empty.
When $F$ has multiple components, then its sutured characteristic is the sum of $\chi_{sut}$ of its components.

\medskip

A sutured surface $F$ without circle sutures can be viewed as a cobordism from $\bdry_- F$ to $\bdry_+ F$, built from attaching $1$--handles to a collar of $\bdry_-F$ and a collection of $0$--handles followed by attaching $2$--handles as needed.  Let $h^i(F)$ is the minimum number of $2$--dimensional $i$--handles used in such a handle structure.
\begin{lemma}\label{lem:sutchar}
For a sutured surface without circle sutures and an associated minimal handle structure
\[  \chi_{sut}(F)  =  -\chi(F) + h^0(F) + h^2(F) + \frac12 \#\bdry_v F = h^1(F).\]
\end{lemma}

\begin{proof} 
It suffices to prove this lemma for connected $F$.

Since $\bdry_0 F = \emptyset$ by assumption, $d(F)$ counts how many of $r_+$ and $r_-$ are empty.  If $r_- = \emptyset$, then there must be $0$--handles used in its handle structure.  Since $F$ is connected, a minimal handle structure has just one $0$--handle.  Similarly if $r_+= \emptyset$, then a minimal handle structure has just one $2$--handle.  If neither of $r_+$ or $r_-$ are empty, then only $1$--handles are needed.
This gives the first equality.

Next, observe that any component of $\bdry F$ contains an even number of interval sutures.  One could glue these in pairs to make an oriented surface $\hatF$ without sutures.  Then we may count that 
\[ \chi(\hatF) = \chi(F) - \frac12 \#\bdry_v F  = h^0(F) - h^1(F) + h^2(F), \]
which implies the second equality.
\end{proof}

\begin{lemma}\label{lem:thurstyleqcob}
Let $F$ be a sutured surface. 
Then
$\chi_-(F) \leq \chi_{sut}(F)$.
\end{lemma}
\begin{proof}
Recall that (i) $\chi_-$ of a connected surface is just the maximum of $-\chi$ and $0$ and (ii) $\chi_-$ is additive on disjoint unions.  Therefore we may restrict attention to the case that $F$ is connected.   If $\chi(F) \leq 0$ then the result is immediate since $d(F)$, $\# \bdry_v F$, and $\# \bdry_0 F$ are all non-negative.
If $\chi(F) >0$ so that $\chi_-(F) = 0$, then $F$ is either a sphere $S^2$ or disk $D^2$ since $F$ is orientable.  Since $S^2$ is closed, it has $d = 2$.  So $\chi_{sut}(S^2)=0$.   For the disk however, $\chi_{sut}$ depends on the suture structure.   Let $D^2_n$ denote the disk with $2n$ interval sutures, for $n \geq 0$. Then one finds that $\chi_{sut}(D^2_0) = 0$ while $\chi_{sut}(D^2_n) = n-1$ if $n\geq 1$.  Lastly, a disk whose boundary is a circular suture also has $\chi_{sut} = 0$.
\end{proof}

Any decomposing surface 
$F$ in a sutured manifold $(M, \gamma)$ inherits the structure of a sutured surface:  $r_\pm = \bdry F \cap R_\pm$ while the sutures of $F$ are the intersections of $\bdry F$ with the sutures of $(M,\gamma)$.  
Thus we may define 
\[\chi_{(M,\gamma)}(F) = \chi_{sut}(F)\]
where $F$ has the induced sutured structure.
Then, for any class $\xi \in H_2(M,\bdry M)$, we may define its {\em sutured characteristic} to be
\[\chi_{(M,\gamma)}(\xi) = \min_{F \in \xi} \{\chi_{(M,\gamma)}(F)\}\]
where the minimum is taken among decomposing surfaces $F$ that represent $\xi$.

\begin{remark}
Since $\chi_{sut}$ is additive on disjoint unions of surfaces, it follows that for any class $\xi \in H_2(M,\bdry M)$ \[\chi_{(M,\gamma)}(n \xi) = |n| \chi_{(M,\gamma)}(\xi) \mbox{ for all integers } n >0. \] 
\end{remark}

\begin{lemma}\label{lem:thurstyleqcobchar}
For a sutured manifold $(M,\gamma)$, the Thurston norm of a 
class $\xi \in H_2(M, \bdry M)$ is a lower bound for its sutured characteristic.  That is, $x(\xi) \leq \chi_{(M,\gamma)}(\xi)$.
\end{lemma}
\begin{proof}
Suppose a properly embedded surface $F$ in $(M,\gamma)$ represents $\xi$ so that $\chi_{(M,\gamma)}(\xi) = \chi_{(M,\gamma)}(F)$.  Then by Lemma~\ref{lem:thurstyleqcob}, $\chi_{(M,\gamma)}(F) \geq \chi_-(F)$.  Since $\chi_-(F) \geq x([F]) = x(\xi)$, the result follows.
\end{proof}

\subsection{A Lower Bound}

Recall, we say a surface $F$ representing a regular class $\xi$ {\em realizes} $h(M,\gamma,\xi)$ if there is a surface $S$ in $M$ such that $h(M,\gamma,\xi) = h(M,F,S)$.
Then the 
{\em handle 
sutured characteristic} of a 
regular class $\xi$ is
\[ \chi_{(M,\gamma)}^{h} (\xi ) = \min_{F \in \xi} \{\chi_{(M,\gamma)}(F) \vert F \mbox{ realizes } h(M,\gamma,\xi)\}.\]

By Lemma~\ref{lem:homologicallyessential}, $h(M,\gamma,0)$ is realized by a splitting with $F=\emptyset$.  So for the trivial class $\xi=0$, we have $\chi_{(M,\gamma)}^{h}(0) = 0$ since $\chi_{sut}(\emptyset) = 0$.

\begin{theorem}
\label{thm:generallowerbound}
Let $(M, \gamma)$ a sutured manifold  without toroidal sutures.
Then for any $\xi \in H_2(M, \bdry M)$
the following inequality holds:

\[ h(M,\gamma,0) - 2d  \leq h(M,\gamma,\xi) + 2 \chi_{(M,\gamma)}^{h}(\xi) \]
where $d=0,1,2$ depending on whether neither, just one, or both of $R_+(\gamma)$ and $R_-(\gamma)$ are empty.
\end{theorem}

\begin{proof}
Since $\chi_{(M,\gamma)}^{h}(\xi) \geq 0$ for all $\xi$,
 the conclusion is trivially satisfied if $\xi=0$. So assume $\xi\neq 0$.

Given  $(M, \gamma, \xi)$, let $F$ be a decomposing surface representing  $\xi$ and realizing $h(M,\gamma, \xi)$ with  $\chi_{(M,\gamma)}(F) = \chi_{(M,\gamma)}^{h}(\xi)$. In $M$, let $F \times [-\epsilon, \epsilon]$ be a collar of $F$, so that $M' = M\cut F$ with the sutured manifold decomposition $(M,\gamma) \overset{F}{\leadsto} (M',\gamma')$.
Recall, we identify $F$ with $F\times \{0\}$ and set $F_{\pm}= F\times \{\mp \epsilon \}$.  With this arrangement, the normal of the oriented surface  $F_+$ points out of $M\cut F$ and of $F_-$ points into $M\cut F$. Next choose the surface $S$ so that $(M,F,S)=(A, B; F,S)$ a circular Heegaard splitting  with handle number $h(M,F,S)=h(A)+h(B)$ realizing $h(M,\gamma, F)$, where: 
\begin{itemize}
    \item $\bdry_-A = R'_- = F_- \cup (R_-\cut F)$  
    \item $\bdry_+A=\bdry_-B= S$
    \item $\bdry_+B = R'_+ = F_+ \cup (R_+\cut F)$
\end{itemize}

For each $\pm = +,-$, define the collars $C_\pm = R_\pm \times I$ and $C'_\pm = R'_\pm \times I$ so that $\bdry R_+ \times I \cup \bdry R_- \times I = A(\gamma)$ and $\bdry R'_+ \times I \cup \bdry R'_- \times I = A(\gamma')$.

\medskip

Let us first assume that no component of $\bdry F$ is contained in the sutures.  Thereafter we will address the necessary modifications when some component of $\bdry$ is contained in the sutures.

\medskip

{\bf Case 1:} {\em No component of $\bdry F$ is contained in the sutures.}\\
Let us also fix a minimal handle structure on $F$, viewed as a cobordism from $F \cap R_-$ to $F \cap R_+$ due to its embedding in the sutured manifold $(M,\gamma)$. 
In particular, a component of $F$ has a $0$--handle only if it is disjoint from $R_-$, and it has a $2$--handle only if it is disjoint from $R_+$.
Furthermore $h^1(F) = \chi_{sut}(F) = \chi_{(M,\gamma)}(F)$.

First observe that $A$ has a compression body structure consisting of $1$--handles attached to $C'_-= R'_- \times I$.  (No component of $A$ is a handlebody since the circular splitting realizes $h(M,\gamma,\xi)$ for $\xi \neq 0$; see Section~\ref{sec:handlecounts}.) 
Since $R'_- = (R_- \cut F) \cup F_-$, let us first thicken $F_-$ by attaching $F\times [0,\epsilon]$ to $A$ and then viewing the collar $C'_-$ as being $(R_- \cut F) \times I \cup F \times [0,\epsilon]$.  Now we may drag the feet of the $1$--handles of $A$ so that each is attached to the portion of the collar over $R_-\cut F$ or a thickened disk neighborhood of a $0$--handle of $F_-$ in $F \times [0,\epsilon]$.  Recall that if a component of $F$ has no $0$--handles in its handle structure  it meets $R_-$; hence the foot of a $1$--handle of $A$ over such a component could be dragged to any component of $R_- \cut F$ that it meets.

Next, consider the manifold $\bar{A} = (A \cup F\times [0,\epsilon]) \cut C_+$ obtained by scalloping out whatever intersects the collar $C_+$ of $R_+$.  Such intersections are just thickenings of  $(\bdry_+ F) \times [0,\epsilon]$ and their removal does not affect the aforementioned handle structure of $A$ or the handle structure of $F$.
Now let us consider the handle structure of $\bar{A}$ induced from the $1$--handles of $A$ and the handles of $F \times [0,\epsilon] \cut C_+$ built upon only $(R_- \cut F) \times I$ and the thickened $0$--handles of $F$, $\calH^0(F) \times [0,\epsilon]$.  Indeed, any $2$--dimensional $k$--handle of $F$ thickens to a $3$--dimensional $k$--handle of $F\times [0,\epsilon]$ and also of $F \times [0,\epsilon] \cut C_+$.  If a component of $F$ has no $0$--handles, then its thickened $1$--handles will be attached to the components of $(R_- \cut F) \times I$ that it meets; otherwise its thickened $1$--handles will be attached to its thickened $0$--handles.  Observe that the $2$--handles of $\bar{A}$ are the thickened $2$--handles of $F$, specifically $\calH^2(F) \times [0,\epsilon]$.

At this point, it is helpful to acknowledge that the analogous dual construction may be carried out for the compression body $B$.  That is, we obtain a handle structure on the manifold $\bar{B} = (B \cup F \times [-\epsilon,0]) \cut C_-$ built upon $(R_+ \cut F) \times I$ and the thickened $2$--handles of $F$, $\calH^2(F) \times [-\epsilon, 0]$.  Cutting out $C_-$ just removes thickenings of $(\bdry_- F) \times [-\epsilon,0]$.  And any $2$--dimensional $k$--handle of $F$ thickens to a $3$--dimensional $k+1$--handle of $F \times [-\epsilon,0]$ which is a dual $2-k$--handle.   
Indeed, the $1$--handles of $\bar{B}$ (the dual $2$--handles) are the thickened $0$--handles of $F$, specifically $\calH^0(F) \times [-\epsilon,0]$.

So now we may consider that $\bar{A} \cup C_-$ only adds to $\bar{A}$ the thickenings of  $(\bdry_- F) \times [-\epsilon,0]$ to complete $(R_- \cut F) \times I$ to $C_- = R_- \times I$. Hence $\bar{A} \cup C_-$ has the handle structure of $\bar{A}$ built upon $C_-$ and the $0$--handles $\calH^0(F) \times [0,\epsilon]$. By deleting the $2$--handles $\calH^2(F) \times [0,\epsilon]$, this becomes a compression body, though one that is possibly disconnected. Similarly $\bar{B} \cup C_+$ is made into a possibly disconnected compression body built upon $C_+$ and the dual $0$--handles $\calH^2(F) \times [-\epsilon,0]$ by deleting its dual $2$--handles, $\calH^0(F) \times [-\epsilon,0]$.  Though as these dual $2$--handles are just $1$--handles, we may add them to the other compression body.
Ultimately, we form the compression bodies
\[A^* = \bar{A} \cup C_- \cup (\calH^0(F) \times [-\epsilon,0]) \cut (\calH^2(F) \times [0,\epsilon])\]
and
\[B^* = \bar{B} \cup C_+ \cup (\calH^2(F) \times [0,\epsilon]) \cut (\calH^0(F) \times [-\epsilon,0]).\]
Since $A^*$ and $B^*$ are compression bodies with disjoint interiors where $M=A^*\cup B^*$, $\bdry_- A^* = R_-$, and $\bdry_+ B^* = R_+$, they define a Heegaard splitting (and Heegaard surface $S^* = \bdry_+ A^* = \bdry_- B^*$) for the connected sutured manifold $M$ and are hence connected as well.  See Figure~\ref{fig:circular-to-lineal} for a schematic of this construction.

\begin{figure}
    \centering
    \includegraphics[width=\textwidth]{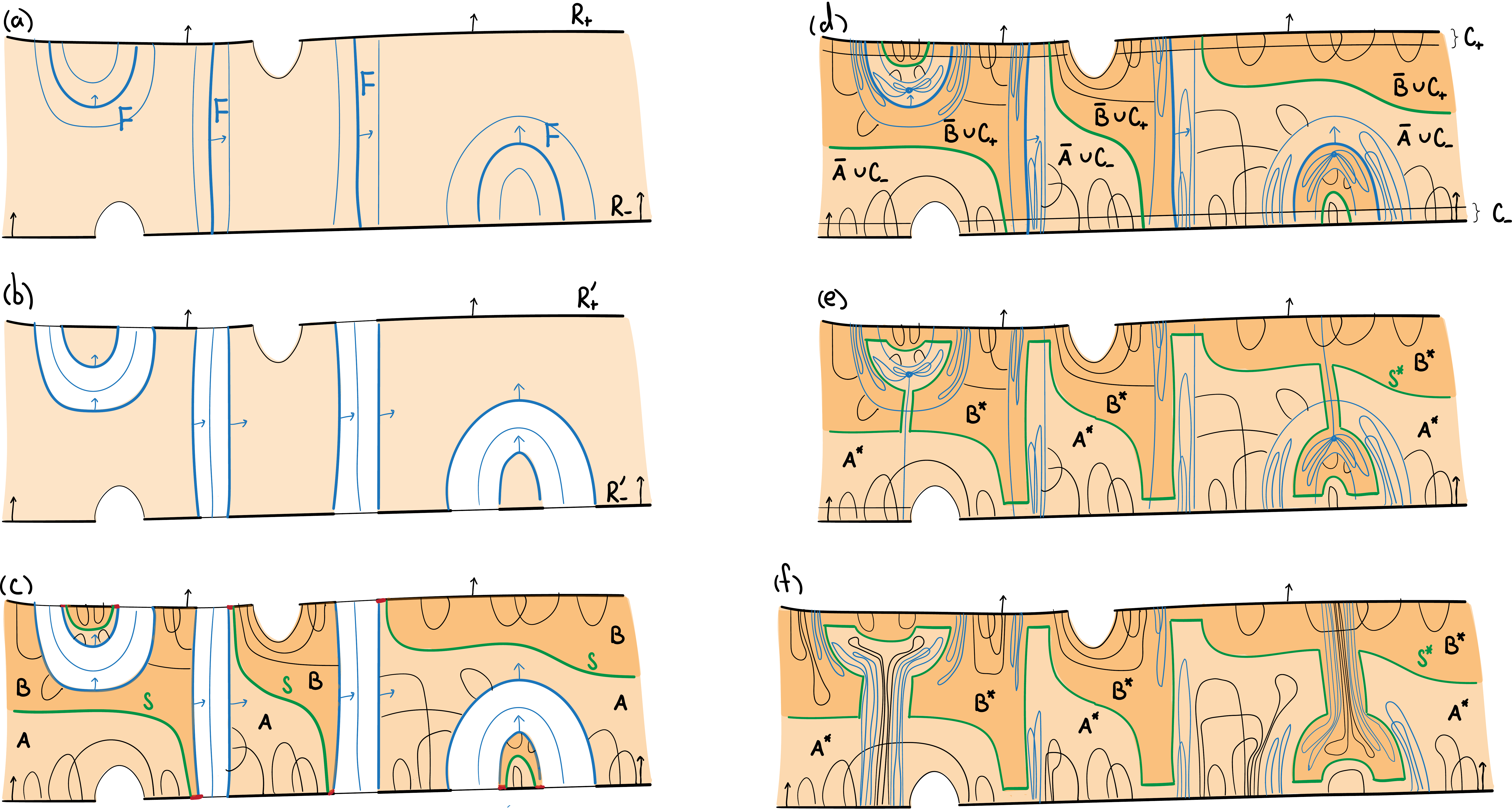} 
    \caption{Schematic pictures for the proof of Theorem \ref{thm:generallowerbound}.  
    (a) The sutured manifold $M$ with the surface $F$ and a collar is indicated.  
    (b) The sutured manifold $M' = M\cut F$ is obtained by decomposing $M$ along $F$. 
    (c) The Heegaard surface $S$ for $M'$, the compression bodies $A$ and $B$, and the cores of their $1$--handles and dual $1$--handles are shown.
    (d)  Halves of the collar of $F$ are added back to $A$ and $B$ to reform $M$, the collars $C_+$ of $R_+$ and $C_-$ of $R_-$ are removed from $A$ and $B$ respectively and then added to the other to form the manifolds $\overline{A} \cup C_-$ and $\overline{B} \cup C_+$.  In the halves of the collar of $F$, thickened handle decompositions of $F$ are now indicated.
    (e) Reassigning $2$--handles of $\overline{A} \cup C_-$ and $1$--handles of $\overline{B} \cup C_+$ associated to handle structures of $F$ induce compression bodies $A^*$ and $B^*$.
    (f) The feet of the $1$--handles of $A^*$ and dual $1$--handles of $B^*$ are pulled to lie in $R_-$ and $R_+$ respectively.
    }
    \label{fig:circular-to-lineal}
\end{figure}

As built, $A^*$ has $h^0(F)$ $0$--handles and $h^1(A) + h^1(F) + h^0(F)$ $1$--handles.  These $0$--handles may each be cancelled with a $1$--handle, unless $R_- = \emptyset$ in which case $A^*$ must be a handlebody and one $0$--handle must remain.  This induces $A^*$ with a handle structure consisting of $h^0(A^*)$ $0$--handles and $h^1(A)+h^1(F)+h^0(A^*)$ $1$--handles.  Similarly all of the dual $0$--handles of $B^*$ may be cancelled with dual $1$--handles, unless $R_+ = \emptyset$ in which case $B^*$ is a handlebody and one dual $0$--handle must remain.
Therefore we obtain a Heegaard splitting $(M,A^*,B^*)$ with handle number
\begin{align*}
    h(M,A^*,B^*) &= h(A^*) + h(B^*)\\
        &= (h^0(A^*) + h^1(A^*)) + (h^0(B^*) + h^1(B^*) )\\
        &= (h^0(A^*) + h^1(A)+h^1(F)+h^0(A^*)) + ( h^0(B^*) + h^1(B) + h^1(F) + h^0(B^*) )\\
        &= h^1(A)+h^1(B) + 2h^1(F) + 2(h^0(A^*)+h^0(B^*)) \\
        &= h(M,\gamma,\xi) + 2\chi_{sut}(F) + 2d.
\end{align*}
Therefore 
\[ h(M,\gamma,0) \leq h(M,\gamma,\xi) + 2\chi_{(M,\gamma)}(F) + 2d.\]
Thus, since $\chi_{(M,\gamma)}(F) = \chi_{(M,\gamma)}^{h}(\xi)$,  we have the desired inequality.

\bigskip

{\bf Case 2:} {\em Some component of $\bdry F$ is contained in the sutures.}\\
We will perturb the surface $F$ to a surface $\tilde{F}$ such that
\begin{itemize}
    \item $\tilde{F}$ is a decomposing surface representing $\xi$ and realizing $h(M, \gamma, \xi)$, 
    \item $\chi_{(M,\gamma)}(\tilde{F}) = \chi_{(M,\gamma)}(F)$, and
    \item no component of $\bdry \tilde{F}$ is contained in the sutures.
\end{itemize}
Then we may apply the argument of Case 1 to $\tilde{F}$ to get the desired result.

For each component of $A(\gamma)$ that contains a component of $\bdry F$, isotop $F \cup S$ in a collar of the annular suture (taken in $(M,\gamma)$) so that afterwards each component of $\bdry F$ and $\bdry S$ in this collar now meets the annular suture in a pair of spanning arcs and each $R_+$ and $R_-$ in a single arc.  See the top row of  Figure~\ref{fig:perturbationcollardecomp}.   Let $\tilde{F}$ and $\tilde{S}$ be the surfaces that result from these isotopies. Since $\tilde{F}$ is isotopic to $F$, it too represents $\xi$.

Observe that since each circular suture of $F$ has been perturbed into a pair of vertical sutures, the perturbed surface $\tilde{F}$ now has no circular sutures while $\chi_{(M,\gamma)}(\tilde{F}) = \chi_{(M,\gamma)}(F)$.  

\begin{figure}
    \centering
    \includegraphics[width=\textwidth]{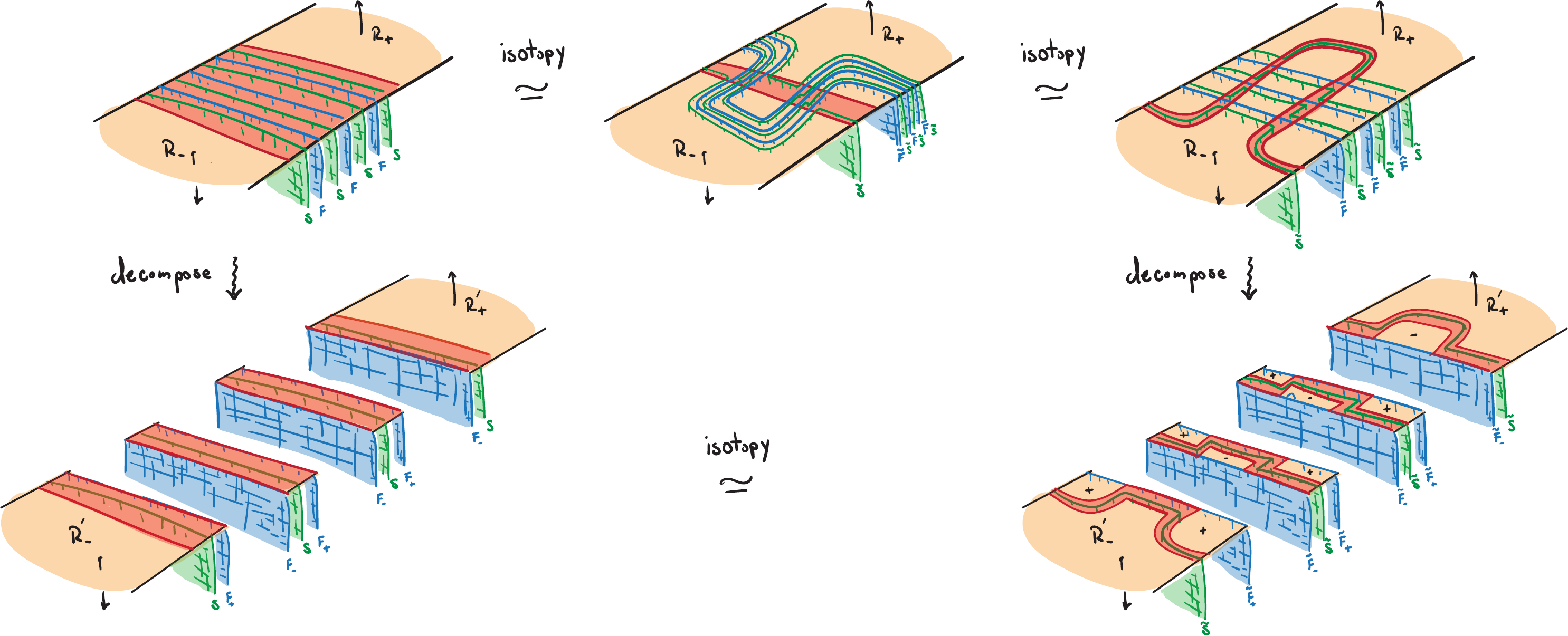}
    \caption{The top row starts with the isotopy of the surfaces $F$ and $S$ in a neighborhood of an annular suture that contains components of $\bdry F$ (and hence also $\bdry S$) to the perturbed surfaces $\tilde{F}$ and $\tilde{S}$.  Thereafter an ambient isotopy straightens out $\tilde{F}$ and $\tilde{S}$ at the expense of perturbing the annular suture.
    The bottom row shows that the decomposition along $F$ and how it contains $S$ is isotopic to the decomposition along $\tilde{F}$ and how it contains $\tilde{S}$.}
    \label{fig:perturbationcollardecomp}
\end{figure}

Moreover, as the perturbation of $(M, F, S)$ to $(M, \tilde{F}, \tilde{S})$ occurs within a collar of $A(\gamma)$, these splittings are basically equivalent: Shaving off a collar of $A(\gamma)$ from each $(M,F,S)$ and $(M,\tilde{F}, \tilde{S})$ leaves equivalent manifolds.  Ultimately, we just need to compare the decompositions along each $F$ and $\tilde{F}$ in this collar.   Figure~\ref{fig:perturbationcollardecomp} shows that these decompositions leave isotopic sutured manifolds $(M \cut F, \gamma')$ and $(M \cut \tilde{F}, \tilde{\gamma})$  with surfaces $S$ and $\tilde{S}$ that are equated by such an isotopy.
\end{proof}

\begin{theorem}\label{thm:bothbounds}
Let $(M, \gamma)$ a sutured manifold  without toroidal sutures.
Then for any $\xi \in H_2(M,\bdry M)$ we have
\[ h(M,\gamma, \xi) \leq h(M,\gamma,0)-2d  \leq h(M,\gamma, \xi) + 2\chi^{h}_{(M,\gamma)}(\xi)\]
where $d=0,1,2$ depending on whether neither, just one, or both of $R_+(\gamma)$ and $R_-(\gamma)$ are empty.
\end{theorem}


\section{Bounds for knots and links}

\subsection{Tunnel numbers and Morse Novikov numbers of links}
\label{sec:tunnelnumber}

Let $M$ be a connected, compact oriented $3$--manifold with non-empty boundary.  A {\em tunnel system} for $M$ is a properly embedded collection of arcs $\tau$ so that $M-\nbhd(\tau)$ is a handlebody.  The {\em tunnel number} $\TN(M)$ of $M$ is the minimum cardinality of its tunnel systems. Equivalently, if we regard $M$ as a connected sutured manifold $(M,\gamma_+)$ for which $\bdry M = R_+$, then the tunnel number of $M$ is the minimum of $h(B)$ among Heegaard splittings $(A,B;S)$ for $M$. 
This gives the equality
\[\TN(M) + g(M,\gamma_+) + 1= h(M,\gamma_+, 0).\]

\begin{observation}\label{obs:TNlh}
For manifolds $M$ with non-empty toroidal boundary,
 $h(M,\gamma_+, 0) = 2\TN(M) +  2$.
\end{observation}

\medskip

Let $L$ be a link in a closed $3$--manifold $Y$.
The {\em exterior} of $L$ is the complement $M=Y\cut L$ of a regular neighborhood.  
Let $(M, \gamma_+)$ be obtained by setting $R_+(\gamma_+) = \bdry M$.
The {\em tunnel number} $\TN(L)$ of $L$ is $\TN(M,\gamma_+)$.

Further suppose $L$ is an oriented, null-homologous link in $Y$.  
Let $(M,\gamma_0)$ be obtained by setting $T(\gamma_0)=\bdry M$.
Then its Seifert surfaces define a regular class $\xi \in H_2(M,\bdry M)$. 
The {\em Morse-Novikov number} $\MN(L)$ of $L$ is 
 $\MN(L) = h(M,\gamma_0, \xi)$.  (Note that this definition can be extended to rationally null-homologous links as well.)
By  Theorem~\ref{thm:reassign}, $h(M,\gamma_0, \xi) = h(M,\gamma_+,\xi)$.  Hence we also have $\MN(L) = h(M,\gamma_+, \xi)$.

\subsection{Bounds for knots and links in $S^3$.}
Let $L$ be an oriented link in $S^3$ of $|L|$ components. A Seifert surface for $L$ is an oriented surface $F$ possibly disconnected but without closed components such that $\bdry F = L$.  Let $M=S^3 \cut L$ be the exterior of $L$, and consider the sutured structure  $(M, \gamma_+)$ obtained by setting $R_+(\gamma_+) = \bdry M$.

\begin{cor}[{\cite[Theorem 1.1]{P}}]\label{cor:mntn}
The Morse-Novikov number of an oriented link $L$ in $S^3$ is at most twice its tunnel number.
That is
\[ \MN(L) \leq 2\TN(L).\]
\end{cor}

\begin{proof}
By Observation~\ref{obs:TNlh}, $h(M,\gamma_+,0) = 2\TN(L) + 2$. 
Now for the homology class $\xi$ of a Seifert surface for $L$ restricted to $(M,\gamma_+)$, Theorem~\ref{thm:lineartocircular} implies that $h(M,\gamma_+,\xi)\leq h(M,\gamma_+,0)-2$. 
Since $\MN(L)= h(M,\gamma_+,\xi)$, this gives our result.
\end{proof}

\begin{lemma} \label{lem:spanningsurf}
Suppose a surface $F$ realizes $h(M,\gamma_+,\xi)$ and $\chi_{(M,\gamma_+)}^h(\xi)$.
Then $F$ is a (possibly disconnected) Seifert surface for $L$ and
\[\chi_{(M,\gamma_+)}(\xi) =2g(F) +|L|-|F|.\]
\end{lemma}

\begin{proof}
Since $M$ is the exterior of a link in $S^3$, any closed component of $F$ homologically trivial in $H_2(M,\bdry M)$.  Then by Lemma~\ref{lem:homologicallyessential}, $h(M,\gamma_+,\xi)$ is realized by a subsurface of $F$ without closed components.  Since $F$ already realizes $\chi_{(M,\gamma_+)}^h(\xi)$, it must have no closed components. By applications of Lemma~\ref{lem:productdecompgivesnondecreasinghandlenumber} along annuli, we may ensure $\bdry F$ has no pairs of parallel but oppositely oriented components in any component of $\bdry M$ while maintaining $\chi_{(M,\gamma_+)}^h(\xi)$.  Then since $\xi$ is a Seifert class, $\bdry F$ must have a single component in each component of $\bdry M$.  Hence $F$ is a Seifert surface for $L$.

As a properly embedded surface in $(M,\gamma_+)$, $F$ inherits the structure of a sutured surface with  $\bdry F =\bdry_+ F = r_+$ so that $h^0(F)=|F|$, $h^2(F)=0$, and $\#\bdry_v(F) = 0$.  Hence by Lemma~\ref{lem:sutchar}, \[\chi_{(M,\gamma_+)}(F) = h^1(F) = |F|-\chi(F).\]
If $F$ has components $F_1, \dots, F_{|F|}$, 
then 
\[\chi(F) = \sum \chi(F_i) = \sum \left(2- 2g(F_i) -|\bdry F_i|\right) = 2|F|-|\bdry F| - 2\sum g(F_i).\]
So we have 
\[
    \chi_{(M,\gamma_+)}(F)=|F|-\chi(F) = |F| - \left(2|F|-|\bdry F| - 2\sum g(F_i)\right) = 2g(F) + |L|-|F|
\]
since $|\bdry F| = |L|$.
\end{proof}

Define $g_\MN(L) = g(F)$ for a surface $F$ that realizes $h(M,\gamma_+,\xi)$ and $\chi_{(M,\gamma_+)}^h(\xi)$.  If $F$ has no closed components, then $F$ is a Seifert surface for $L$.

\begin{cor}\label{cor:tnmn}
Let $L$ be an oriented link in $S^3$.  
Then
\[2\TN(L) \leq \MN(L) + 4g_{\MN}(L).  \]
\end{cor}

\begin{proof}
Let $\xi$ be the class of a Seifert surface for $L$ realizing $h(M,\gamma_+,\xi)$ and $\chi_{(M,\gamma_+)}^h(\xi)$. 
By Theorem~\ref{thm:generallowerbound},
\[ h(M,\gamma_+,0) - 2  \leq h(M,\gamma_+,\xi) + 2 \chi_{(M,\gamma_+)}^{h}(\xi) \]
because $d=1$ as $R_+ = \bdry M$.
By definition, $h(M,\gamma_+,\xi) = \MN(L)$.

By Observation~\ref{obs:TNlh} $h(M,\gamma_+,0) = 2\TN(L) + 2$. 
Lemma~\ref{lem:spanningsurf} and the definition of $g_\MN(L)$ imply that  $\chi_{(M,\gamma_+)}^{h}(\xi) = 2g_\MN(L)$.  Hence $2\TN(L) \leq \MN(L) + 4g_{\MN}(L)$.
\end{proof}

\begin{lemma}
If $L$ is a link in $S^3$, then $\MN(L)$ is realized by a connected Seifert surface.
\end{lemma}

\begin{proof}
Suppose a surface $F$ realizes $\MN(L)=h(M,\gamma_+,\xi)$ where $\xi$ is the Seifert class for $L$.  By Lemma~\ref{lem:homologicallyessential}, we may assume no component of $F$ is closed. Thus $F$ is non-separating. So now if $F$ is not connected, then by Lemma~\ref{lem:handleext} there is a sequence of tubings of $F$ that produces a connected surface that also realizes $\MN(L)$.
\end{proof}

\begin{theorem}\label{thm:bothboundslinks}
For an oriented link $L$ in $S^3$,
\[ \MN(L) \leq 2\TN(L) \leq \MN(L) + 4g_{\MN}(L). \]
\end{theorem}

\begin{proof}
This follows directly from Corollaries~\ref{cor:mntn} and \ref{cor:tnmn}.
\end{proof}

\begin{cor}\label{cor:tnoffiberedlinks}
If $L$ is a fibered link then
$\TN(L) \leq 2g(L)$.
\end{cor}

\begin{proof}
If $L$ is fibered, then $\MN(L) = 0$ and $g_\MN(L) = g(L)$.
Then Theorem~\ref{thm:bothboundslinks} implies the result.
\end{proof}

\begin{cor}\label{cor:bounds}
Suppose $\{K_k\}_{k\in \N}$ is a family of knots on which $\TN$ is unbounded.
Then at most one of $g_{\MN}$ and $\MN$ is bounded.
\end{cor}

\begin{proof}
By Theorem~\ref{thm:bothboundslinks}, $\MN(K_k) + 4g_{\MN}(K_k) \geq 2\TN(K_k)$.
\end{proof}

\begin{cor}\label{cor:EML}
There is a family of genus one knots $\{K_k\}_{k\in \N}$ for which either $g_{\MN}$ is unbounded or $\MN$ is unbounded.
\end{cor}

\begin{proof}
By \cite[Theorem 3.7(2)]{EML}, the knots $K_n$ described in \cite[Figure 3]{EML} have $\TN(K_n)\geq n$.  One readily observes that these knots are all genus one knots.  Now apply Corollary~\ref{cor:bounds}.
\end{proof}

\begin{remark}
We expect that the knots $K_n$ described in \cite[Figure 3]{EML} have a unique incompressible Seifert surface, except possibly for small $n$.  If so, then $g_{MN}(K_n)=g(K_n)=1$ for those knots.   Corollaries~\ref{cor:bounds} and \ref{cor:EML} would then imply that $\{\MN(K_n)\}$ is unbounded.  
\end{remark}

\begin{remark}
Hirasawa-Rudolph asked for examples of knots where $\MN(K) > 2g(K)$ \cite{HirasawaRudolph}.  The only explicit examples known so far are certain pretzel knots with $g=1$ and $\MN = 4$  \cite{freegenusone}.  These knots also have $g_\MN = 1$.
\end{remark}

\begin{conj}[Cf. {\cite[Remark p440]{CTPforKnots}}]
\label{conj:gmn=g}
For an oriented link $L$ in $S^3$, $g_{\MN}(L) = g(L)$.
\end{conj}

\section{The circular handle number function on homology}
\label{sec:extension}

\subsection{Regular classes}
Recall that for a sutured manifold $(M,\gamma)$, a  class $\xi \in H_2(M,\bdry M)$ is {\em regular} if every surface representing $\xi$ meets every toroidal suture. A class is {\em non-regular} if it is not regular.  To phrase this more homologically, let $T(\gamma) \subset \bdry M$ be the toroidal sutures of $(M,\gamma)$.  Then a class $\xi$ is {\em non-regular} if $\bdry \xi \in H_1(\bdry M)$ is zero on some component of $T(\gamma)$.  This second version of the definition extends to other coefficients.  In particular, we may now speak of regular classes and non-regular classes of $H_2(M,\bdry M; \R)$.

\begin{lemma}\label{lem:nonregclasses}
Let $(M,\gamma)$ be a connected sutured manifold.
Unless $M$ is a rational homology sphere, the set of non-regular classes of $H_2(M,\bdry M; \R)$ is a finite union of subspaces of positive codimension.
\end{lemma}

\begin{proof}
Since $H_2(M,\bdry M; \R) = 0$ exactly when $M$ is a rational homology sphere, let us assume $H_2(M,\bdry M; \R) \neq 0$.  
If $T(\gamma)=\emptyset$, then every non-zero class is regular and the result holds.  So now further assume $T(\gamma) \neq \emptyset$.

Let $T$ be a non-empty subset of the components of $T(\gamma)$. Then we claim that the set $V_{T}$ of classes of $H_2(M,\bdry M; \R)$ whose boundary vanishes on $T$ form a subspace.  By definition, the members of $V_{T}$ are necessarily non-regular classes.  
So now suppose $\sigma$ and $\tau$ are non-regular classes where $\bdry \sigma$ and $\bdry \tau$ are zero on $T$.
Since $\bdry (r \sigma + s \tau) = r \bdry \sigma + s \bdry \tau$ for any $r,s \in \R$, $\bdry (r \sigma + s \tau)$ is  zero on $T$.  Hence $r \sigma + s \tau \in V_{T}$.  Thus $V_{T}$ is a subspace.

Consequently, letting $\mathcal{T}$ be the collection of all non-empty subsets of components of $T(\gamma)$, the set of non-regular classes of $H_2(M,\bdry M; \R)$ is $\bigcup_{T \in \mathcal{T}} V_T$.  Since $T(\gamma)$ is finite, $\mathcal{T}$ is finite, so this is a finite union.

Finally, we claim that there exists a regular class.  By the Half-Lives Half-Dies Lemma (that half of the rank of $H_1(\bdry M;\Z)$ dies upon the inclusion of $\bdry M$ into $M$), for each component of $\bdry M$  there is a primitive class $\eta \in H_2(M,\bdry M; \Z)$ for which $\bdry \eta$ is non-zero on that component. In particular, for each torus $T_i$ of $T(\gamma)$ let $\eta_i$ be such a primitive class (where $i$ runs from $1$ to $n=|T(\gamma)|$).  Then for each $k$ from $1$ to $n$ we may successively find a non-zero integral coefficient $c_k$ so that the class $\zeta_k = \sum_{i=1}^k c_i \eta_i$ has $\bdry \zeta_k$ non-zero on each of $T_1, \dots, T_k$.  Thus $\zeta_n$ will be a regular class.
\end{proof}

\begin{lemma}\label{lem:regclasses}
For a connected sutured manifold $(M,\gamma)$, 
the set of regular classes of $H_2(M,\bdry M;\R)$ is a dense open set.
\end{lemma}

\begin{proof}
This follows from the non-regular classes being a finite union of subspaces of positive codimension, Lemma~\ref{lem:nonregclasses}.
\end{proof}


\subsection{A function on homology}
Fix a sutured manifold structure $(M,\gamma)$ on a $3$--manifold $M$. To begin, let us write $h(\xi) = h(M,\gamma,\xi)$ for an integral regular class $\xi \in H_2(M, \bdry M; \Z)$.    Theorem~\ref{thm:chfunction} examines properties of this function.  We extend $h$ to real and non-regular classes in Theorem~\ref{thm:chfunctionext}.

\begin{theorem}\label{thm:chfunction}
For a sutured manifold $(M, \gamma)$,
the handle number function 
\[h \colon H_2(M,\bdry M) \to \N\]
has the following properties for any regular classes $\alpha, \beta \in H_2(M, \bdry M)$: 
\begin{enumerate}
    \item $h(\beta) \leq N$ for some constant $N$ (depending on $(M,\gamma)$),
    \item $h(k \beta) = h(\beta)$ for any integer 
    $k >0$, and
    \item $h(\alpha+m\beta) \leq h(\beta)$ for sufficiently large $m$.
\end{enumerate}
\end{theorem}

\begin{proof}[Proof of Theorem~\ref{thm:chfunction}]

(1) This follows immediately from Theorem~\ref{thm:lineartocircular}  
(and Theorem~\ref{thm:reassign} in case 
 $(M,\gamma)$ has toroidal sutures).

(2) 
Let $k$ be a positive integer.  If $h(k\beta) = h(\beta)$ holds for primitive homology classes, then it holds for non-primitive ones too.  So assume $\beta$ is a primitive homology class.

Say $F$ realizes $h(\beta)$.  Then let $F_k$ be $k$ parallel copies of $F$ (in a collar neighborhood of $F$ disjoint from $S$), and let $S_k$  be $S$ along with  $k-1$ copies of $F$ interlaced between the previous stack. Then the Heegaard splitting $(M,F_k, S_k)$ has the same handle number as $(M,F,S)$.  Since $[F_k] = k[F] = k \beta$, $h(\beta) \geq h(k\beta)$.

Now for the other direction.  Among surfaces realizing $h(k \beta)$, say $\tilde{F}$ is one with the least number of components. By Theorem~\ref{thm:wellconditioned} and its proof, we may assume $\tilde{F}$ is well-conditioned. In particular, $[\tilde{F}] = k \beta$, there is a Heegaard splitting $(M, \tilde{F}, \tilde{S})$ with $h(M, \tilde{F}, \tilde{S}) = h(M, \tilde{F}) =  h(k\beta)$, $\tilde{F}$ satisfies condition (C1), and no other such surface has fewer components.
For each connected compression body $A_i$ of the compression body $\tilde{A}$ of the Heegaard splitting $(M, \tilde{F}, \tilde{S}) = (\tilde{A},\tilde{B};\tilde{F}, \tilde{S}) $, let $\tilde{F}_i = \tilde{F} \cap A_i$.  Observe that each $\tilde{F}_i$ also satisfies condition (C1) since it is a union of components of $\tilde{F}$.

If $\tilde{F}_i$ has just one component for each $i$, 
then the components of $M \cut \tilde{F}$ provide the homology between pairs of components of $\tilde{F}$.  So they all must represent the same integral homology class, say $\beta'$. Also, let $k'=|\tilde{F}|$ be the number of these components.   So then we have 
$k\beta = [\tilde{F}] = [\cup_i \tilde{F}_i] = \sum_i [\tilde{F}_i] = k' \beta'$.  Since $\beta$ is a primitive class, the Representation Theorem of \cite{MP} implies that $\tilde{F}$ has at least $k$ components, ie. $k' \geq k$. Hence $\beta = \frac{k'}{k} \beta'$ in $H_2(M,\bdry M;\Q)$.  Since $\beta$ is a primitive class and $\beta'$ is an integral class, it follows that in fact $k'=k$ and $\beta'=\beta$.  That is, each $\tilde{F}_i$ represents the primitive class $\beta$.
In this case, let $\tilde{F}_1$ be one of these components.  Then by amalgamations along all other components (Lemma~\ref{lem:amalg}), we obtain a 
circular Heegaard splitting $(M, \tilde{F}_1, \tilde{S}_1)$ (where $\tilde{S}_1$ is the surface obtained from the amalgamations) that represents $h(\beta)$.

On the other hand, suppose that, for say $i=1$, we have $|\tilde{F}_1| \geq 2$.
If $A_1$ had a minimal handle structure with a $1$-handle whose feet were in two  components of $\tilde{F}_1$, then by Lemma~\ref{lem:handleext} a tubing along the handle would create a Heegaard splitting $(M,\tilde{F}^*, \tilde{S}^*)$ realizing $h(k \beta)$ but with $|\tilde{F}^*| < |\tilde{F}|$ contrary to assumptions.  Hence the components of $\tilde{F}_1$ must all lie in the same component of $\bdry_- A_1$, and 
$A_1$ is formed from $\bdry_-A_1 \times I$ with just enough $1$--handles to connect these components of $\bdry_- A_1$.  (So  $h(A_1) = |\bdry_-A_1| -1$.)

 Say $Q_0$ is the component of $\bdry_- A_1$ containing $\tilde{F}_1$. Then $Q_0 \cut \tilde{F}_1$ must be a collection of components of $R_- \cut \tilde{F}$ 
 which meet $\tilde{F}_-$.  (If a component of $R_-\cut \tilde{F}$ does not meet $\tilde{F}_-$, then it does not belong to a component of $\bdry_- A$ that meets $\tilde{F}$.)
Since $\tilde{F}_1$ has at least $2$ components,  $\tilde{F}_1 \subset Q_0$, and $Q_0$ is connected, some component $Q$ of $Q_0 \cut \tilde{F}_1$ must meet least two components of $\tilde{F}_1$. Let $R_0$ be the component of $R_-$ containing this component $Q$.
Then Lemma~\ref{lem:dcsumheegsplitting} produces a new decomposing surface $\tilde{F}_{R_0}'$ homologous to $\tilde{F}$ and maintaining that $h(M,\tilde{F}_{R_0}') = h(M, \tilde{F}) =  h(k\beta)$, but with fewer components.  However, this is contrary to our assumptions.

(3) 
We will construct a splitting $(M,F,S)$ with $h(M,F,S) \leq h(\beta)$ such that $F$ represents $\alpha+m\beta$ for some $m >0$.

Suppose $h(\beta)$ is realized by the  Heegaard splitting $(M,F_\beta, S_\beta)$ which splits $M$ into the compression bodies $A_\beta$ and $B_\beta$.  
By Lemma~\ref{lem:homologicallyessential} we may assume that $F_\beta$ is well-conditioned.
Hence no component of either $A_\beta$ or $B_\beta$ is a handlebody. (If a component were a handlebody, then its boundary would be a component of $S_\beta$ which is a Heegaard surface for a component of $M \cut F_\beta$.  The collection of components of $F_\beta$ meeting that component of $M \cut F_\beta$ would be null homologous.)
Suppose the surface $F_\alpha$ represents $\alpha$. Isotop $F_\alpha$ to both minimally intersect and be transverse to $F_\beta \cup S_\beta$ and $\gamma$.

Let $(M_i, \gamma_i)$ for $i=1,\dots,n$ be the components of the result of the sutured manifold decomposition of $(M,\gamma)$ along $F_\beta$.  Then for each $i$, the surface $S_i = S_\beta \cap M_i$ is a Heegaard surface for $(M_i, \gamma_i)$.  Let $F_i = F_\alpha \cap M_i$.  

The proof of Theorem~\ref{thm:lineartocircular} produces a new surface $F_i'$ in $\nbhd(F_i \cup S_i)$ that is homologous there to $F_i \dcsum m_i S_i$ for sufficiently large $m_i >0$, satisfies $F_i' \cap R_\pm(\gamma_i) = F_i \cap R_\pm(\gamma_i)$, and  decomposes the compression bodies $A_i$ and $B_i$ of $M_i \cut S_i$ into compression bodies $A_i'$ and $B_i'$ of the same handle numbers.  

Choose $m$ to be larger than all of the $m_i$, and then update each $F_i'$ be the surface as above but using $m-1$ instead of $m_i$. Update $A_i'$ and $B_i'$ as well.
Then we may set $F' = \cup F_i'$ so that $[F']=[F]+(m-1)[S_\beta] = \alpha + (m-1) \beta$.
Next form the surfaces $F_0 = F' \dcsum F_\beta$ and $S_0 = F' \dcsum S_\beta$, and note that $[F_0] = [S_0] =  \alpha + m \beta$.
(Indeed, $[F_0] = [F']+[F_\beta] =\alpha + (m-1) \beta + \beta = \alpha + m \beta$.) 
Observe now that 
$(M,F_0,S_0)$ is a circular Heegaard splitting with compression bodies $A_0 = \bigcup  A_i \cup (F' \cap B_i)\times I$ and $B_0 = \bigcup B_i \cup (F'\cap A_i)\times I$.
Since $h(A_0) = \sum h(A_i) = h(A)$ and $h(B_0) =\sum h(B_i) = h(B)$, we find that $h(M, F_0, S_0) = h(M, F_\beta, S_\beta)$ which is $h(\beta)$.   Hence $h(\alpha+m\beta) \leq h(\beta)$.
\end{proof}

\begin{theorem}\label{thm:chfunctionext}
The  handle number extends to a function $h \colon H_2(M,\bdry M; \R) \to \N$ that is
\begin{enumerate}
    \item bounded by constant,
    \item constant on rays from the origin, and
    \item locally maximal.
\end{enumerate}
\end{theorem}

\begin{proof}
Let $\beta$ be a non-zero regular class.
Since $h(k \beta) = h(\beta)$ for any positive integer $k$, we may define $h(r \beta) = h(\beta)$ for any positive real number $r$.  Hence $h$ is constant on rays from the origin in $H_2(M,\bdry M; \R)-\{0\}$ that meet integral regular classes.

Let $S$ be the quotient space of rays from the origin in $H_2(M,\bdry M; \R)-\{0\}$.  Say a point in $S$ is rational if its ray contains an integral class; otherwise say it is irrational. Similarly, say a point in $S$ is regular if its ray belongs to the regular classes of $H_2(M,\bdry M;\R)$. We just showed that $h$ is well-defined on the  rational regular points of $S$.   Observe that property (3) of Theorem~\ref{thm:chfunction} implies that $h$ is locally maximal on the  rational regular points of $S$.  That is, for any neighborhood $U$ of a given  rational regular point $b \in S$ that consists of regular points, if $a$ is a rational point $U$ then $h(a) \leq h(b)$.  We may use this property to extend $h$ to the irrational regular points of $S$ and hence to the rest of the regular classes of $H_2(M,\bdry M; \R)$.

Let $z$ be an irrational regular point of $S$.  Then define $h(z)$ so that:
For any neighborhood $U$ of $z$ consisting of regular points, there exists a neighborhood $U' \subset U$ of $z$ for which $h(z)$ is the maximum of $h$ among the rational points in $U'$.
Then we can see that $h$ is locally maximal on the regular points of $S$.  

Since the regular classes of $H_2(M,\bdry M;\R)-\{0\}$ are a dense open set by Lemma~\ref{lem:regclasses}, the regular points of $S$ also form a dense open set of $S$.  Thus we may extend $h$ to irregular points similarly to how we extended $h$ to irrational regular points.  Let $z$ be an irregular point of $S$ and define $h(z)$ as follows:  For any neighborhood $U$ of $z$, there is a neighborhood $U' \subset U$ of $z$ for which $h(z)$ is the maximum of $h$ among the regular points in $U'$.  As before, we similarly see that $h$ is locally maximal on $S$.

Consequently, with $h$ defined on $S$, we may extend $h$ to all of $H_2(M,\bdry M; \R)-\{0\}$ so that it is constant on rays.  
If $0$ is a regular class, then $h(M,\gamma,0)$ is defined and  $h(M,\gamma,\xi) \leq h(M,\gamma,0)-2d$ for any regular class $\xi$ by Theorem~\ref{thm:generallowerbound}.  Thus it follows that $h(0) \geq h(\xi)$ for any class $\xi \in H_2(M,\bdry M;\R)$ since $d\geq0$.  
If $0$ is not a regular class, then we may just define $h(0)$ to be the maximum of $h(\xi)$ among the regular classes $\xi$.   By Theorem~\ref{thm:chfunction}(1) and our extension of $h$, this maximum exists.  Hence, in either case, $h$ is bounded by $h(0)$.
\end{proof}

\begin{remark}
It may  happen that $h(M,\gamma, \xi)  \neq h(M, \gamma, -\xi)$.   Flipping the orientation of a decomposing surface $F$ may significantly change the resulting suture structure of the decomposed manifold.
\end{remark}

\begin{remark}\label{rem:productsm}
As shown in the proof of Theorem~\ref{thm:chfunctionext}, when $0$ is a regular class $h(\xi)$ is bounded above by the linear handle number of $(M,\gamma)$.  Consequently, the handle number function of a product sutured manifold is the constant function $h=0$.
\end{remark}


\section{Further remarks and questions.}

\subsection{Regarding handle numbers and faces of the Thurston norm}\label{sec:fiberedfaces}

\begin{remark}\label{rem:fiberedfaces}
For integral classes $\beta \in H_2(M,\bdry M;\Z)$, we have $h(\beta)=0$ precisely when there is a surface representing $\beta$ that decomposes $(M,\gamma)$ into a product sutured manifold.    In light of Theorems~\ref{thm:chfunction} and \ref{thm:chfunctionext}, one may regard any class $\beta \in H_2(M,\bdry M; \R)$ with $h(\beta)=0$ as a {\em fibered class} with respect to the sutured manifold structure $(M,\gamma)$.

Indeed, if $\bdry M$ is a union of tori and $\gamma$ has no annular sutures, then for $\beta \neq 0$ we have $h(\beta) = 0$ if and only if $\beta$ is a fibered class in the traditional sense (ie. $\beta$ lies in the cone over a fibered face of the Thurston norm).  When $\gamma$ has annular sutures we will have $h(\beta)>0$ for such a fibered class if $\bdry \beta$ has the same slope as an annular suture in some component of $\bdry M$.
\end{remark}

\begin{question} 
Let $(M,\gamma_+)$ be a sutured manifold where $\bdry M$ is a union of tori and $R_+(\gamma_+)=\bdry M$.
Is the handle number constant on any face of the Thurston norm ball?
\end{question}

\begin{remark}
If $M$ has tunnel number $1$, then this indeed is the case since the handle number of any regular class is bounded above by $2$ as discussed in Observation~\ref{obs:TNlh}. Since the handle number is $0$ exactly on the fibered faces (cf.\ Remark~\ref{rem:fiberedfaces}), it must be $2$ elsewhere.  Hence the handle number is constant on any face.

Exteriors of two-bridge links provide a class of such examples.  Following \cite[Theorem 28]{CompGenSatTorti}, one may determine all the fibered faces of the Thurston norm of a two-bridge link exterior and hence the handle number function.  Note that some two-bridge links, such as  $[0; 4,4,-2]$, have no fibered class, so their handle number functions are the constant function $h=2$.
\end{remark}

\begin{remark}
Let $(M,\gamma)$ be a sutured manifold where $\bdry M$ is a union of tori and $R_+(\gamma)$ is a union of essential annuli.  In particular each component of $\bdry M$ has some annular sutures of some slope.  Then the handle number function of $(M,\gamma)$ need not be constant on faces of the Thurston norm.
\end{remark}

\subsection{Handle complexities of 2nd homology classes} \label{sec:handlecomplexities}
The Thurston norm $x(\xi)$ and sutured characteristic $\chi_{(M,\gamma)}(\xi)$ measure the complexity
of any class $\xi \in H_2(M,\bdry M; \Z)$ of the sutured manifold $(M,\gamma)$ by minimizing the complexity $x$ or $\chi_{(M,\gamma)}$ among surfaces representing $\xi$.  
(Recall that for an oriented surface $F$, $x(F) = \chi_-(F)$ which is the sum of $-\chi(F_0)$ among the components $F_0$ of $F$ which are neither spheres nor disks.) 
Restricting to regular classes, we may create  
variants $x^{h}$ and $\chi_{(M,\gamma)}^{h}$ of these, the {\em handle characteristic} and {\em handle sutured characteristic} respectively,  by minimizing these complexities among surfaces representing a regular class $\xi$ that also realize $h(M,\gamma,\xi)$.

\begin{question}
Does there always exist a surface $F$ that realizes both $x^{h}(\xi)$ and $\chi_{(M,\gamma)}^{h}(\xi)$?
\end{question}

These definitions immediately imply that $x(\xi) \leq x^{h}(\xi)$ and $\chi_{(M,\gamma)}(\xi) \leq \chi^{h}_{(M,\gamma)}(\xi)$.   Lemma~\ref{lem:thurstyleqcob} implies that
$ x(\xi) \leq \chi_{(M,\gamma)}(\xi)$ and $x^{h}(\xi) \leq \chi_{(M,\gamma)}^{h}(\xi)$.
Together, we have:
\[\begin{tikzcd}
	{\chi_{(M,\gamma)}(\xi)} & {\chi_{(M,\gamma)}^{h}(\xi)} \\
	{x(\xi)} & {x^{h}(\xi)}
	\arrow["\rotatebox{90}{$\leq$}"{description}, draw=none, from=2-1, to=1-1]
	\arrow["\rotatebox{90}{$\leq$}"{description}, draw=none, from=2-2, to=1-2]
	\arrow["\leq"{description}, draw=none, from=1-1, to=1-2]
	\arrow["\leq"{description}, draw=none, from=2-1, to=2-2]
\end{tikzcd}\]

For a knot exterior $M = S^3 \cut \nbhd(K)$ we may take a sutured manifold structure $\gamma_+$ with $\bdry M = R_+$.  Then $x(F) = \chi_{(M,\gamma_+)}(F)-1=2g(F)-1$ for any Seifert surface $F$.
So for the Seifert class $\xi=[F]$, we have $x(\xi) = \chi_{(M,\gamma_+)}(\xi)-1 = 2g(K)-1$. It further follows that $x^{h}(\xi) = \chi^{h}_{(M,\gamma_+)}(\xi)-1$.
Hence, in this situation, Conjecture~\ref{conj:gmn=g} that   
$g(K) = g_{MN}(K)$  asks if either of the horizontal inequalities above --- and hence both --- are equalities.
We may extend Conjecture~\ref{conj:gmn=g} to other sutured manifolds.
\begin{question}\label{ques:handlecharvsthurstonnorm}
For a non-trivial regular class $\xi \in H_2(M,\bdry M;\Z)$ of a sutured manifold $(M,\gamma)$:
\begin{itemize}
    \item Does $\chi_{(M,\gamma)}(\xi) = \chi^{h}_{(M,\gamma)}(\xi)$?
    \item Does $x(\xi) = x^{h}(\xi)$?
\end{itemize}
\end{question}

While $x$ extends to a (pseudo)norm on $H_2(M,\bdry M;\R)$, it is not clear that $x^{h}$ behaves so well.
\begin{question}
Does the triangle inequality hold for $x^{h}$?  That is, for non-trivial regular integral classes $\alpha$ and $\beta$ such that $\alpha+\beta$ is regular, do we have
\[x^{h}(\alpha + \beta) \leq x^{h}(\alpha)+x^{h}(\beta)?\]
\end{question}
Recall that we develop the extension of $h(M,\gamma,\xi)$ to non-trivial regular classes in $H_2(M,\bdry M;\R)$ and then complete it to all classes in Section~\ref{sec:extension}.  So if there are situations where the triangle inequality holds for $x^{h}$, does it extend to a (pseudo)norm?


\bibliographystyle{alpha}
\bibliography{biblio}

\end{document}